\documentclass[10pt,letterpaper]{article}
\usepackage{amssymb}
\usepackage{amsmath}
\usepackage{amsrefs}
\usepackage{amsthm}

\usepackage{mathabx} 
\usepackage{array}
\usepackage{rotating}
\usepackage{verbatim}
\usepackage{enumitem} 

\newenvironment{proof_sketch}{\noindent{\it Sketch of
    Proof.}\hspace*{1em}}{\qed\bigskip} 
\theoremstyle{plain}
\newtheorem{theorem}[equation]{Theorem}
\newtheorem{corollary}[equation]{Corollary}
\newtheorem{lemma}[equation]{Lemma}

\newtheorem{proposition}[equation]{Proposition}

\theoremstyle{definition}

\theoremstyle{remark}

\newtheoremstyle{indenteddefinition}{\topsep}{\topsep}{\addtolength{\leftskip}{2.0em}}{-0em}{\bfseries}{.}{
}{} 
\theoremstyle{indenteddefinition}

\DeclareMathOperator\rank{rank}
\DeclareMathOperator\Ad{Ad}

\DeclareMathOperator\Aut{Aut}

\DeclareMathOperator\Hom{Hom}

\DeclareMathOperator\HT{ht}

\DeclareMathOperator\End{End}

\DeclareMathOperator\tr{tr}
\DeclareMathOperator\gr{gr}

\DeclareMathOperator\Sig{Sig}
\DeclareMathOperator\Spin{Spin}

\newcommand\Tstrut{\rule{0pt}{2.6ex}}
\newcommand\Bstrut{\rule[-0.9ex]{0pt}{0pt}} 

\begin{document}
\title{Signatures for finite-dimensional representations of real reductive Lie groups}
\author{Daniil Kalinov 
  \\Department of
  Mathematics \\ MIT
\and David A. Vogan, Jr.\\Department of Mathematics\\ MIT\\
\and Christopher Xu \\MIT}
\begin{abstract}
We present a closed formula, analogous to the Weyl dimension formula,
for the signature of an invariant Hermitian form on any
finite-dimensional irreducible representation of a real reductive Lie
group, assuming that such a form exists. The formula shows in a
precise sense that the form must be very indefinite. For example, if
an irreducible representation of $GL(n,R)$ admits an invariant form of
signature $(p,q)$, then we show that $(p-q)^2 \le p+q$. The proof is
an application of Kostant's computation of the kernel of the Dirac
operator.
\end{abstract}
\date{\today}
\maketitle

\section{Introduction}\label{sec:intro}
\setcounter{equation}{0}
\begin{subequations}\label{se:introsig}
Suppose $G$ is a complex connected reductive algebraic group defined
over ${\mathbb R}$, and $G({\mathbb R})$ the group of real
points. Suppose that
\begin{equation}
  (\pi,V), \qquad \pi\colon G({\mathbb R}) \rightarrow GL(V) \simeq
  GL(\dim(\pi),{\mathbb C})
\end{equation}
is an irreducible finite-dimensional complex representation of
$G({\mathbb R})$. Of course Weyl's dimension formula provides a simple
closed formula for $\dim(\pi)$. It often happens that $V$ admits a
non-zero $G({\mathbb R})$-invariant Hermitian form
\begin{equation}
  \langle,\rangle_\pi\colon V \times V \rightarrow {\mathbb C}, \qquad
  \langle \pi(g)v,\pi(g)w\rangle = \langle v,w\rangle.
\end{equation}
In this case Schur's lemma guarantees that the form is non-degenerate,
and unique up to a nonzero real factor. Sylvester's law says that the
form has a signature
\begin{gather}
(p(\pi),q(\pi)), \quad p(\pi) + q(\pi) = \dim(\pi), \\ \pi\colon G
  \rightarrow U(V,\langle\cdot,\cdot\rangle) \simeq U(p(\pi),q(\pi)).
\end{gather}
Changing the form by a positive factor does not change $p(\pi)$ and
$q(\pi)$, and changing it by a negative factor interchanges
them. Therefore both the absolute value of the difference and the
unordered pair
\begin{equation}
  \Sig(\pi) =_{\text{\textnormal{def}}} |p(\pi) - q(\pi)|, \qquad
  \Sigma(\pi) =_{\text{\textnormal{def}}} \{p(\pi),q(\pi)\}
\end{equation}
are well-defined whenever $\pi$ is finite-dimensional irreducible, and
admits a non-zero invariant form. Because $\dim(\pi)$ is computable,
calculating $\Sigma(\pi)$ is equivalent to calculating the
non-negative integer $\Sig(\pi)$. That calculation is the main result
of this paper (Theorem \ref{thm:Gsig}), with a formula nearly as easy
to calculate as Weyl's dimension formula. Because the general case
involves a number of slightly subtle technicalities, we will in this
introduction state only a special case.
\end{subequations} 

\begin{theorem}\label{thm:GLsig} Suppose $G=GL(n,{\mathbb R})$, and
  $$\lambda=(\lambda_1,\ldots,\lambda_n)$$
  is a decreasing sequence of integers. Write
  $$(\pi_{\mathbb C}(\lambda),V(\lambda)) = \text{algebraic
    representation of $GL(n,{\mathbb C})$ of highest weight
    $\lambda$}.$$
  Write
  $$n=2m+\epsilon, \qquad m=[n/2], \qquad \epsilon = 0\ \text{or}\ 1.$$
  \begin{enumerate}
    \item The restriction $\pi(\lambda)$ to
      $GL(n,{\mathbb R})$ is still irreducible.
      \item The representation $\pi(\lambda)$ of $GL(n,{\mathbb R})$
        admits an invariant Hermitian form if and only if
        $$\lambda = (-\lambda_n, -\lambda_{n-1},\ldots,-\lambda_1);$$
        equivalently, if there is a decreasing sequence of nonnegative
        integers
        $$\mu = (\mu_1,\ldots,\mu_m)$$
        so that
        $$\lambda = \begin{cases}
          (\mu_1,\ldots,\mu_m,-\mu_m,\ldots,-\mu_1) & (\epsilon = 0)\\
          (\mu_1,\ldots,\mu_m,0,-\mu_m,\ldots,-\mu_1) & (\epsilon = 1)
        \end{cases}$$
 \item Suppose $\pi(\lambda)$ admits an invariant Hermitian
   form. Define $\sigma(\mu)$ to be the irreducible representation of
   $\Spin(n)$ of highest weight $\mu+(1/2,\ldots,1/2)$. Then
   $$\Sig(\pi(\lambda)) = \dim(\sigma(\mu))/2^{m-1+\epsilon}.$$
  \end{enumerate}
\end{theorem}
The denominator in the last formula is the dimension of an irreducible
(half) spin representation of $\Spin(n)$, of highest weight
$(1/2,\ldots,1/2)$. That it always divides the numerator is a
classical fact about representations of spin groups. Of course the
division is needed to make the formula give the correct signature of
$+1$ in case $\lambda=0$.

This formulation is a bit misleading. The general result Theorem
\ref{thm:Gsig} involves for $GL(n)$ a rather different representation of
$\Spin(n)$, of highest weight
$$2\mu + (m-1+ \epsilon/2, m-3+\epsilon/2,\ldots,\epsilon/2).$$
The proof of Theorem \ref{thm:GLsig} will then follow by a formal manipulation
of the Weyl dimension formula. We carry out the details at the end of
Section \ref{sec:dirac}.

\begin{subequations}\label{se:GLsig}
Nevertheless we can see in this special case some interesting behavior
of the signature. In what follows we use the notation of the theorem,
always assuming that
\begin{equation}
  \pi(\lambda)= \text{finite-dimensional Hermitian irreducible of
    $GL(n,{\mathbb R})$.}
\end{equation}
Because $SL(n,{\mathbb R})$ is noncompact and
simple, it cannot admit nontrivial finite-dimensional unitary
representations; that is, there can be no nontrivial homomorphism from
$SL(n,{\mathbb R})$ to $U(N)$. Consequently $\Sig(\pi(\lambda)) > 0$ whenever
$\lambda\ne 0$. It is not difficult to prove (for example, using the
structure of maximal tori in $U(p,q)$) a little more: a
nontrivial homomorphism from $SL(n,{\mathbb R})$ to $U(p,q)$ can exist
only if $|p-q| \ge n-1$.  That is,
$$\Sig(\pi(\lambda)) \ge n-1 \qquad (\lambda \ne 0).$$
This estimate is the best possible absolute bound, because
$$\Sig(\pi(1,0,\ldots,0,-1)) = n-1$$
either by Theorem \ref{thm:GLsig} or by direct calculation of the
invariant Hermitian form
$$\langle X,Y\rangle = \tr(X\overline{Y})$$
on the complexified adjoint representation (on $n\times n$ complex
matrices of trace zero).

One thing that Theorem \ref{thm:GLsig} shows is the ``typical'' behavior of
signatures. The Weyl dimension formula is a polynomial in $\lambda$:
\begin{equation}
  \deg_\lambda(\dim(\pi(\lambda))) = \left(n^2 - n\right)/2 =
  \binom{n}{2} = 
2m^2 + m(2\epsilon - 1).
\end{equation}
(The number of positive roots for $G$ is $(\dim G - \rank G)/2$.) The
signature formula in the theorem is also a polynomial in $\lambda$,
but now of degree
\begin{equation}
  \deg_\lambda(\Sig(\pi(\lambda))) = \left(\binom{n}{2} - [n/2]\right)/2 = m^2
  +m(\epsilon-1) \le \deg_\lambda(\dim)/2.
\end{equation}
The conclusion is that for ``generic'' $\lambda$,
\begin{equation}\label{eq:growth}
  \text{signature grows more slowly than square root of dimension:}
\end{equation}
the invariant Hermitian form is close to being maximally
isotropic. There is a similar statement for any real reductive $G({\mathbb R})$,
with square root replaced by
$$(\dim K-\rank(K))/(\dim G - \rank(G)).$$
Here $K({\mathbb R})$ is a maximal compact subgroup of $G({\mathbb
  R})$.

For $GL(n,{\mathbb R})$, the formulas are so simple and explicit
that we can calculate
\begin{equation}
  \dim(\pi(\lambda)) = \Sig(\pi(\lambda))^2 \cdot \prod_{i=1}^m
  \frac{2\lambda_i +n-2i+1}{n-2i+1}
\end{equation}
This is a much stronger version of \eqref{eq:growth}. It
would be fascinating to find a direct representation-theoretic
interpretation of this formula. (A difficulty is that for $n\ge 4$,
the last product (which is always at least $1$) need not be an
integer.)
\end{subequations} 

Here is how the rest of the paper is organized.  Section
\ref{sec:ident} recalls the highest weight parametrization of
finite-dimensional representations (Proposition \ref{prop:fin}) and
the identification of Hermitian representations (Proposition
\ref{prop:herm}). Section \ref{sec:resweyl} concerns the structure of
the ``restricted Weyl group;'' it can be omitted in the (very common)
case $\rank G = \rank K$. Section \ref{sec:hwtform} calculates the
signature of an invariant Hermitian form on extremal weight spaces
(Corollary \ref{cor:Wupper1}); this is easy, amounting to a
calculation in $SL(2,{\mathbb R})$. Section \ref{sec:dirac} recalls
from \cite{Kmult} and \cite{HKP}*{Theorem 4.2} elementary facts
(Proposition \ref{prop:Kmult}) about the eigenvalues of the Dirac
operator on finite-dimensional representations. A simple linear
algebra result (Lemma \ref{lemma:selfadjsig}), based on the
self-adjointness of the Dirac operator, then implies that the
signature of an invariant Hermitian form is essentially equal to the
signature on the kernel of the Dirac operator (Corollary
\ref{cor:D2}). Finally, we use the result from Section
\ref{sec:hwtform} to calculate the signature on the kernel of the
Dirac operator, and deduce our main result Theorem \ref{thm:Gsig}
calculating signatures for finite-dimensional representations of
arbitrary real reductive groups.

We thank Jeffrey Adams for pointing out to us the interesting behavior
of signatures of forms on finite-dimensional representations. The
third author, who is an MIT undergraduate student, embarked on an
exploration of this behavior using the {\tt atlas} software from
\cite{atlas} as a summer research project in 2018, under the guidance
of the first author, an MIT graduate student.  He discovered
experimentally the polynomial dependence on $\lambda$ in Theorem
\ref{thm:GLsig}. The first author found a way to bound signatures from
above, which for $GL(n,{\mathbb R})$ gave the formula in Theorem
\ref{thm:GLsig} as an upper bound for $\Sig(\pi(\lambda))$. (This
method of the first author is a version of Lemma
\ref{lemma:selfadjsig}.) At this point the second author, who was old
enough to remember \cite{Kmult}, was able to join the race at Hereford
Street.

\section{Weights and Hermitian representations}
\label{sec:ident}
\setcounter{equation}{0}
\begin{subequations}\label{se:realred}
  We continue as in \eqref{se:introsig} with
  \begin{equation}\label{eq:realred}\begin{aligned}
      G \quad &\text{complex connected reductive algebraic group}\\
      \sigma_{\mathbb R} \colon G \rightarrow G \quad &\text{antiholomorphic
        involutive automorphism}\\
      G({\mathbb R}) = G^{\sigma_{\mathbb R}} \quad &\text{real form of
        $G$}.
    \end{aligned}
  \end{equation}
  We will make constant use of a fixed {\em Cartan involution}
  \begin{equation}\label{eq:cartan}
    \theta\colon G \rightarrow G\quad \text{algebraic involutive
      automorphism;}
  \end{equation}
the characteristic requirement of $\theta$ is that the antiholomorphic
automorphism
\begin{equation}\label{eq:cptform}
  \sigma_c =_{\text{\textnormal{def}}} \theta\circ\sigma_{\mathbb R}
\end{equation}
is a compact real form of $G$. Then automatically
\begin{equation}\label{eq:K}
  K =_{\text{\textnormal{def}}} G^\theta
\end{equation}
is a (possibly disconnected) reductive subgroup of $G$,
preserved by $\sigma_{\mathbb R}$, and
\begin{equation}\label{eq:KR}
  K({\mathbb R}) = K^{\sigma_{\mathbb R}} = K^{\sigma_c}
\end{equation}
is a maximal compact subgroup of $G({\mathbb R})$. The {\em Cartan
  decomposition} of the Lie algebra is the eigenspace decomposition
under $\theta$:
\begin{equation}\label{eq:cartandecompA}\begin{aligned}
{\mathfrak s} &=_{\text{\textnormal{def}}} \text{$-1$ eigenspace of
  $\theta$},\\ {\mathfrak g} &= {\mathfrak k} + {\mathfrak
  s},\\ {\mathfrak g}({\mathbb R}) &= {\mathfrak k}({\mathbb R}) +
{\mathfrak s}({\mathbb R}).\\
\end{aligned}\end{equation}
What is much deeper and more powerful and is the Cartan decomposition of the
group:
\begin{equation}\label{eq:cartandecompB}
  G({\mathbb R}) = K({\mathbb R})\cdot \exp({\mathfrak s}({\mathbb
    R}));
\end{equation}
the map from right to left is a diffeomorphism.
\end{subequations} 

\begin{subequations} \label{se:realtori}
Every $\sigma_{\mathbb R}$-stable maximal torus $H\subset G$ has a
$G({\mathbb R})$-conjugate which is preserved by $\theta$. We
therefore consider
\begin{equation}\label{eq:maxtorA}\begin{aligned}
    H &\subset G \quad \text{maximal torus}\\
    \sigma_{\mathbb R}(H) &= H, \qquad \theta(H) = H\\
    H({\mathbb R}) &=_{\text{\textnormal{def}}} H^{\sigma_{\mathbb R}}
    \quad\text{real points of $H$}\\
    T &=_{\text{\textnormal{def}}} H^\theta = H\cap K\\
        T({\mathbb R}) &=_{\text{\textnormal{def}}} H({\mathbb
          R})^\theta = H({\mathbb R})\cap K({\mathbb R})\\
            &= \ \text{maximal compact subgroup of $H({\mathbb R}$}.\\
  \end{aligned}
\end{equation}
Notice that $T$ is a reductive abelian algebraic group, and
$T({\mathbb R})$ its (unique) compact real form. If we define
$${\mathfrak a} = {\mathfrak h}({\mathbb R}) \cap {\mathfrak s},
\qquad A = \exp({\mathfrak a}).$$
then the Cartan decomposition \eqref{eq:cartandecompB} gives a Lie
group direct product decomposition
\begin{equation}\label{eq:maxtorB}
  H({\mathbb R}) = T({\mathbb R}) \times A.
\end{equation}
The group $A$ is {\em not} algebraic: if we define $B=H^{-\theta}$
then $B$ is an abelian algebraic group, and
$$A = \text{Lie group identity component of $B({\mathbb R})$}$$
The notation in \eqref{eq:maxtorB} (particularly for $A$) is very
traditional and rather useful (for describing continuous characters of
$H({\mathbb R})$, for example). But the non-algebraic nature of $A$
must always be remembered.

The roots of $H$ in $G$ are complex-valued algebraic (and in
particular holomorphic) characters of $H$:
\begin{equation}\label{eq:root}
  \alpha \colon H \rightarrow {\mathbb C}^\times \qquad (\alpha \in
  R(G,H)).
\end{equation}
As holomorphic characters, the roots are determined by their
restrictions to $H({\mathbb R})$, or the differentials of those
restrictions:
\begin{equation}
  \alpha_{\mathbb R}\colon H({\mathbb R})\rightarrow {\mathbb
    C}^\times, \qquad d\alpha_{\mathbb R}\colon {\mathfrak h}({\mathbb
    R})\rightarrow {\mathbb C}.
\end{equation}
We will often write just $\alpha$ for either $\alpha_{\mathbb R}$ or
its differential, relying on the context to avoid ambiguity. But for
the structural results we are now describing, it is helpful to
maintain an explicit distinction. In accordance with tradition, we
will write the group structure on roots as $+$, even though it
corresponds to multiplication of characters of $H$.

Because the automorphism $\theta$ is assumed to preserve $H$, it
automatically acts on the roots. A moment's thought shows that
$\sigma_{\mathbb R}$ also permutes the root spaces, and therefore acts
on the roots by the requirement
$$[\sigma_{\mathbb R}(\alpha)](h) =_{\text{\textnormal{def}}}
\overline{\alpha(\sigma_{\mathbb R}^{-1}(h))}.$$ 
These two actions are related by
\begin{equation}\label{eq:sigmatheta}
  \theta(\alpha) = \sigma_{\mathbb R}(-\alpha).
\end{equation}

The root $\alpha$ is called {\em real} if $d\alpha_{\mathbb R}$ is
real valued (equivalently, if $\alpha_{\mathbb R}$ is
real-valued). Because of \eqref{eq:sigmatheta},
\begin{equation}\label{eq:real}
  \text{$\alpha$ is real} \iff \sigma_{\mathbb R}(\alpha) = \alpha
  \iff \theta(\alpha) = -\alpha.
\end{equation}
In case $\alpha$ is real, the root subgroup
$$\phi_{\alpha} \colon SL(2) \rightarrow G$$
may be chosen to be defined over ${\mathbb R}$ with the standard real
form of $SL(2)$:
\begin{equation}
  \phi_{\alpha} \colon SL(2,{\mathbb R}) \rightarrow G({\mathbb R})
\end{equation}

The root $\beta$ is called {\em imaginary} if
$d\beta_{\mathbb R}$ is imaginary-valued (equivalently, if
$\beta_{\mathbb R}$ takes values in the unit circle). Because of
\eqref{eq:sigmatheta},
\begin{equation}\label{eq:imag}
  \text{$\beta$ is imaginary} \iff \sigma_{\mathbb R}(\beta) = -\beta
  \iff \theta(\beta) = \beta.
\end{equation}
In case $\beta$ is imaginary, the root subgroup $\phi_\beta$ is
defined over ${\mathbb R}$, but with one of two different real forms
of $SL(2)$. In case
\begin{equation}
  \phi_{\beta} \colon SU(1,1) \rightarrow G({\mathbb R}),
\end{equation}
we say that $\beta$ is {\em noncompact imaginary}. In case
\begin{equation}
  \phi_{\beta} \colon SU(2) \rightarrow G({\mathbb R}),
\end{equation}
we say that $\beta$ is {\em compact imaginary}.

Finally, the root $\gamma$ is called {\em complex} if
$d\gamma_{\mathbb R}$ is neither real nor purely imaginary valued
(equivalently, if $\gamma_{\mathbb R}$ takes non-real values of
absolute value not equal to $1$). Because of \eqref{eq:sigmatheta},
\begin{equation}\label{eq:cplx}
  \text{$\gamma$ is complex} \iff \sigma_{\mathbb R}(\gamma) \ne \pm\gamma
  \iff \theta(\gamma) \ne \pm \gamma.
\end{equation}
It is equivalent to require that the root subgroup
$$\phi_\gamma \colon SL(2) \rightarrow G$$
is not defined over ${\mathbb R}$ for any real structure on $SL(2)$.
\end{subequations} 

\begin{subequations} \label{se:maxsplit}
There are (up to conjugation by $K({\mathbb R})$) two maximal tori of
particular interest to us. First is the {\em maximally split torus}
\begin{equation}\label{eq:maxsplit}
  H_s({\mathbb R}) = T_s({\mathbb R}) \cdot A_s.
\end{equation}
This torus is characterized by the three equivalent requirements
\begin{equation}\begin{aligned}
  &\text{$\dim_{\mathbb R}A_s$ is as large as possible}\\
    &\text{$\dim_{\mathbb R}T_s({\mathbb R})$ is as small as possible}\\
    &\text{there are no noncompact imaginary roots of $H_s$ in $G$.}
  \end{aligned}
\end{equation}
Inside the root system $R(G,H_s)$ we can find a set of positive roots
$R^+_s = R^+(G,H_s)$ satisfying
\begin{equation}
  \text{the nonimaginary roots in $R^+_s$ are preserved by $-\theta$.}
\end{equation}
There is a unique Weyl group element specified by the requirement
$$w_{0,s}(R^+(G,H_s) = \theta R^+(G,H_s);$$
it commutes with $\theta$ (as an automorphism of $H$), and so
acts on $T({\mathbb R})$ and $A$.
\end{subequations} 

\begin{subequations} \label{se:maxcpt}
Next, the {\em maximally compact torus} (sometimes called the {\em
  fundamental torus})
\begin{equation}\label{eq:maxcpt}
  H_c({\mathbb R}) = T_c({\mathbb R}) \cdot A_c.
\end{equation}
This torus is characterized by the four equivalent requirements
\begin{equation}\begin{aligned}
  &\text{$\dim_{\mathbb R}A_c$ is as small as possible}\\
    &\text{$\dim_{\mathbb R}T_c({\mathbb R})$ is as large as
      possible}\\
    &\text{$T_c({\mathbb R})_0$ is a maximal torus in $K({\mathbb R})_0$}\\
    &\text{there are no real roots of $H_c$ in $G$.}
  \end{aligned}
\end{equation}
Inside the root system $R(G,H_c)$ we can find a set of positive roots
$R^+_c = R^+(G,H_c)$ satisfying
\begin{equation}
  \text{$R^+_c$ is preserved by $\theta$.}
\end{equation}
\end{subequations} 

\begin{subequations}\label{se:realweights}
Our next goal is to recall the parametrization (due to Cartan and
Weyl) of finite-dimensional representations of $G({\mathbb R})$ by highest
weights.  In order to do that, we need two more bits of notation. For
each root $\alpha$, recall that the {\em coroot $\alpha^\vee$} is the
restriction to the maximal torus of the root $SL(2)$:
\begin{equation}
  \alpha^\vee\colon {\mathbb C}^\times \rightarrow H, \quad
  \alpha^\vee(z) = \phi_\alpha\begin{pmatrix} z & 0 \\ 0 &
  z^{-1}\end{pmatrix}.
\end{equation}
The homomorphism $\phi_\alpha$ is unique only up to conjugation by
diagonal matrices in $SL(2)$, but $\alpha^\vee$ is (therefore)
absolutely unique. The homomorphism $\alpha^\vee$ is specified by its
differential
\begin{equation}
  H_\alpha =_{\text{\textnormal{def}}} d\phi_\alpha \begin{pmatrix} 1 & 0 \\ 0 &
    -1\end{pmatrix} \in {\mathfrak h}.
\end{equation}
If $\alpha$ is {\em real}, so that $\phi_\alpha$ is defined over
${\mathbb R}$, then we can define
\begin{equation}
  m_\alpha =_{\text{\textnormal{def}}} \phi_\alpha \begin{pmatrix} -1
    & 0 \\ 0 & -1\end{pmatrix} = \alpha^\vee(-1) = \exp(i\pi H_\alpha)
    \in H({\mathbb R}),
\end{equation}
an element of order (one or) two in the real Cartan subgroup
$H({\mathbb R})$.

A character
\begin{equation}\begin{aligned}
  \gamma&\colon H({\mathbb R}) \rightarrow {\mathbb C}^\times
  \text{continuous}\\
  d\gamma({\mathbb R}) &\colon {\mathfrak h}({\mathbb R}) \rightarrow
  {\mathbb C} \qquad \text{real linear}\\
  d\gamma &\colon {\mathfrak h} \rightarrow
  {\mathbb C} \qquad \text{complex linear}
 \end{aligned} \end{equation}
is called {\em weakly integral} if
\begin{equation}
  d\gamma(H_\alpha) \in {\mathbb Z} \qquad (\alpha \in R(G,H)).
\end{equation}
It is called {\em strongly integral} if it is integral, and also
\begin{equation}\label{eq:strInt}
  \gamma(m_\alpha) = (-1)^{d\gamma(H_\alpha)} \qquad (\alpha \in
  R(G,H)\ \text{real}).
\end{equation}

\end{subequations} 

\begin{proposition}\label{prop:fin} Suppose $G$ is a reductive algebraic group as in
  \eqref{se:realred}, and $H$ is a real $\theta$-stable maximal torus
  as in \eqref{se:realtori}.
  \begin{enumerate}
    \item Every irreducible finite-dimensional representation of
      $G({\mathbb R})$ remains irreducible on restriction to the
      identity component, and so defines an irreducible
      finite-dimensional complex representation of the complex
      reductive Lie algebra ${\mathfrak g}$.
      \item The $H$-weights of finite-dimensional representations of
        $G({\mathbb R})$ are precisely the strongly integral
        characters of $H({\mathbb R})$ (see \eqref{eq:strInt}).
      \item If $\gamma$ is a strongly integral character of
        $H({\mathbb R})$, then there is a finite-dimen\-sion\-al
        representation $F(\gamma)$ of $G({\mathbb R})$ having extremal weight
        $\gamma$.
      \item If $H=H_s$ is maximally split, then the representation
        $F(\gamma)$ is uniquely determined.
  \end{enumerate}
\end{proposition}

Most of this is proven in \cite{Vgreen}*{Section 0.4}.

\begin{corollary}\label{cor:fin} In the setting of Proposition
  \ref{prop:fin}, define
  $$G({\mathbb R})^{[H]} = G({\mathbb R})_0\cdot H({\mathbb R}),$$
  the subgroup of $G({\mathbb R})$ generated by the identity component
  and the fixed maximal torus. Put
  $$\pi_0^H(G({\mathbb R})) = G({\mathbb R})/G({\mathbb R})^{[H]},$$
  the quotient of the component group by the image of the component
  group of $H({\mathbb R})$. Define
  $$G({\mathbb R})^\sharp = \{g\in G({\mathbb R}) \mid \Ad(g) \in
  \Ad(G({\mathbb R})_0\} \supset G_0({\mathbb R}),$$
  $$\pi_0^\sharp(G({\mathbb R})) = G({\mathbb R})/G({\mathbb R})^\sharp.$$
  \begin{enumerate}
    \item Each group $G({\mathbb R})^{[H]}$ contains $G({\mathbb
      R})^\sharp$, with equality for the maximally compact Cartan
      $H=H_c$ of \eqref{eq:maxcpt}.
      \item Each group $\pi_0^H$ is a quotient of $\pi_0^\sharp$, which
        is a finite product of copies of ${\mathbb Z}/2{\mathbb Z}$.
     \item Suppose $\gamma$ is a strongly integral character of
       $H({\mathbb R})$. Then there is a simply transitive action of
       the character group of $\pi_0^H$ on the set of
       finite-dimensional irreducible representations of $G({\mathbb
         R})$ having extremal weight $\gamma$. The action is given by
       tensoring with the irreducible characters of $G({\mathbb
         R})/G({\mathbb R})^{[H]}$.
      \end{enumerate}
  \end{corollary}
\begin{proof} For (1), suppose $g\in G({\mathbb R})^\sharp$. Choose
  (according to the definition of $G({\mathbb R})^\sharp$) $g_0 \in
  G({\mathbb R})_0$ so that $\Ad(g) = \Ad(g_0)$. This means in
  particular that $g_0^{-1}g \in Z(G({\mathbb R})) \subset H({\mathbb
    R})$, which is the first assertion of (1). The last assertion we
  will address in Section \ref{sec:hwtform} after we have discussed
  restricted roots.

  The first assertion in (2) is an immediate consequence of (1). The
  second we will prove in Proposition \ref{prop:semidirect}(6) below.

For (3), Proposition \ref{prop:fin} guarantees that there is an
irreducible finite-di\-men\-sion\-al $F(\gamma)$ of extremal weight
$\gamma$, and that $F(\gamma)$ remains irreducible for $G({\mathbb
  R})_0$. The rest of (3) is a formal consequence.
  \end{proof}

We turn next to the calculation of Hermitian duals.

\begin{proposition}\label{prop:herm} Suppose again that $G$ is a
  reductive algebraic group as in
  \eqref{se:realred}, and $H$ is a real $\theta$-stable maximal torus
  as in \eqref{se:realtori}. We use the decomposition
  $$H({\mathbb R}) = T({\mathbb R}) \times A$$
  of \eqref{eq:maxtorB}. Write
  $$\begin{aligned}
    X^*(T)  &= \{\text{continuous characters}\ \lambda_{\mathbb R}\colon
    T({\mathbb R}) \rightarrow S^1\}\\
&\simeq \{\text{algebraic characters}\  \lambda\colon T
    \rightarrow {\mathbb C}^\times\}\\
&\simeq X^*(H)/(1-\theta)X^*(H)\end{aligned}$$
for the characters of the compact group $T({\mathbb R})$. We identify
characters of the vector group $A=\exp({\mathfrak a})$ with the complex
dual space
$${\mathfrak a}^*_{\mathbb C} = \Hom({\mathfrak a},{\mathbb C}),$$
sending $\nu \in {\mathfrak a}^*_{\mathbb C}$ to the character
$$\exp(X) \mapsto \exp(\nu(X)) \qquad (X\in {\mathfrak a})$$
or equivalently
$$a\mapsto a^\nu \qquad (a \in A).$$
\begin{enumerate}
\item Characters of $H({\mathbb R})$ may be indexed by pairs
$$\gamma = (\lambda,\nu) \in X^*(T) \times {\mathfrak a}^*_{\mathbb
  C}.$$
\item  The differential of such a character $\gamma$ is
  $$d\gamma = (d\lambda,\nu) \in i{\mathfrak t}({\mathbb R})^* \times
  {\mathfrak a}^*_{\mathbb C}.$$
\item The Hermitian dual of $\gamma$ is
  $$\gamma^h = (\lambda, -\overline\nu).$$
\item If $\gamma$ is strongly integral, then the Hermitian dual of a
  finite-dimensional representation $F(\gamma)$ (of extremal weight
  $\gamma$) is a finite-dimensional representation $F(\gamma^h)$ (of
  extremal weight $\gamma^h$).
  \item Suppose $\gamma^h$ is conjugate by $W$ to $\gamma$, so that
    $F(\gamma)^h$ also has extremal weight $\gamma$. If $\pi_0^H$ is
    trivial (see Corollary \ref{cor:fin}) (in
  particular, if $G({\mathbb R})$ is connected) then the (uniquely
  determined) $F(\gamma)$ must
  admit an invariant Hermitian form. If
  $\pi_0^H$ is not trivial, then either all or none of the
  $|\pi_0^H|$ choices for $F(\gamma)$ admits an invariant Hermitian form.
\item Suppose $H_s$ is maximally split as in \eqref{se:maxsplit}, and
  $(\lambda_s,\nu_s)$ is a strongly integral $R^+_s$-dominant
  weight. Then
    $$F(\lambda_s,\nu_s)^h = F(w_{0,s}\cdot\lambda_s,-w_{0,s}\cdot
  (\overline{\nu_s})).$$
  In particular, there is a nonzero invariant
  Hermitian form if and only if
    $$\nu_s = -w_{0,s}\cdot (\overline{\nu_s}), \quad
    w_{0,s}\cdot\lambda_s = \lambda_s.$$
\item Suppose $H_c$ is maximally compact as in \eqref{se:maxcpt}, and
  $(\lambda_c,\nu_c)$ is a strongly integral $R^+_c$-dominant
  weight. Then
    $$F(\lambda_c,\nu_c)^h = F(\lambda_c,-\overline{\nu_c}).$$
  In particular, there is a nonzero invariant
  Hermitian form only if
  $$\text{$\nu_c$ is purely imaginary.}$$
  \end{enumerate}
\end{proposition}
Describing sufficient conditions for the existence of a form using the
maximally compact Cartan $H_c$ is complicated; we will address this in
Corollary \ref{cor:extrsig}.
\begin{proof_sketch} The identification of algebraic characters of $T$
  with continuous characters of the compact real form $T({\mathbb R})$
  is a feature of any reductive algebraic group. Parts (1)--(3) are
  immediate. The Hermitian dual of a direct sum is the direct sum of
  the Hermitian duals, so (4) follows. In case $H=H_s$, the extremal
  weight $(\lambda_s,-\overline{\nu_s})$ is evidently dominant for the
  positive system $w_{0,s}\cdot R^+(G,H_s)$, and (5) follows.

  For (6), the only difficulty is that $F(\lambda_c,\nu_c)$ is not
  unique: we only know that $F(\lambda_c,\nu_c)^h$ is {\em some}
  representation of highest weight
  $(\lambda_c,-\overline{\nu_c})$. This proves the necessity of the
  condition in (6) (for existence of an invariant form). The
  hypothesis for the last assertion in (6) amounts to
  $$G({\mathbb R}) = G({\mathbb R})_0 T_c({\mathbb R}),$$
  which implies that the representation $F(\lambda_c,\nu_c)$ is
  unique.
\end{proof_sketch}

\section{Restricted Weyl group}
\label{sec:resweyl}
\setcounter{equation}{0}
Our goal is to study invariant Hermitian forms on extremal weight
spaces with respect to a maximally compact Cartan
\begin{subequations}\label{se:resweyl}
$$H_c = T_c({\mathbb R})\cdot A_c$$
as in \eqref{se:maxcpt}. In order to do that, we first need to
understand how the roots and Weyl group of $H_c$ restrict to $T_c$;
that is the subject of this section. Fix a $\theta$-stable system of positive
roots
$$R_c^+ \subset R(G,H_c).$$

Write
\begin{equation}
  W = W(G,H_c) = W(R(G,H_c)) \subset \Aut(H_c)
\end{equation}
for the Weyl group of $H_c$ in $G$. We are interested in several
subgroups of $W$, including
\begin{equation}\label{eq:Wc}\begin{aligned}
  W^\theta &= \text{centralizer of $\theta$ in $W$}\\
W_{\text{\textnormal{imag}}} &=
W(R_{c,\text{\textnormal{imag}}}) \subset W^\theta \qquad \text{Weyl
  group of imaginary roots} \\
W_K &= N_K(H)/(K\cap H) \simeq N_K(T)/T \subset W^\theta \\
W_{K_0} &= N_{K_0}(T_0)/T_0 \subset W_K
\qquad \text{compact Weyl group}\\
\end{aligned}\end{equation}
The reason we do not call $W_K$ the ``compact Weyl group'' is that it
need not be the Weyl group of a root system.

The first important fact about the maximally compact Cartan is that
{\em no root is trivial on $T_{c,0}$}. The reason is that (for any
$\theta$-stable real Cartan) the roots vanishing on ${\mathfrak t}$
are exactly the real roots (see \eqref{eq:real}); and on $H_c$ there
are no real roots (see \eqref{se:maxcpt}). We can therefore introduce
the {\em restricted roots}
\begin{equation}\label{eq:res}\begin{aligned}
  R_{\text{\textnormal{res}}}(G,T_{c,0}) &= \{\overline{\alpha} =
  \alpha|_{T_{c,0}} \mid \alpha \in R(G,H_c)\}\\ &\subset X^*(T_0) =
  X^*(H_c)/X^*(H_c)^{-\theta}.
\end{aligned}\end{equation}
The dual lattice to $X^*(H_c)/X^*(H_c)^{-\theta}$ is
\begin{equation}
  X_*(T_{c,0}) = X_*(H_c)^\theta
\end{equation}
The {\em restricted coroots} are by definition
\begin{equation}\label{eq:restypes}
  \overline{\alpha}^\vee = \begin{cases} \alpha^\vee & \text{$\alpha =
    \theta\alpha$ imaginary}\\
    \alpha^\vee+ \theta\alpha^\vee & \text{$\alpha$ complex,
      $\alpha+\theta\alpha$ not a root}\\
    2(\alpha^\vee + \theta\alpha^\vee) & \text{$\alpha$ complex,
      $\alpha+\theta\alpha$ a root.}\\
  \end{cases}
\end{equation}
\end{subequations} 

\begin{proposition}\label{prop:resroots}
  The restricted roots and coroots form a root datum
  $${\mathcal R}_{\text{\textnormal{res}}} =
  (X^*(T_{c,0}),R_{\text{\textnormal{res}}},
  X_*(T_{c,0}),R^\vee_{\text{\textnormal{res}}})$$
in the torus $T_{c,0}$. This root datum is not reduced when the third
case for coroots arises. Restriction to $T_{c,0}$ defines an isomorphism
$$W^\theta|_{T_{c,0}} = W({\mathcal R}_{\text{\textnormal{res}}}).$$
Inside this root datum are several smaller root data.
\begin{enumerate}
  \item The {\em reduced restricted root datum}, written ${\mathcal
    R}_{\text{\textnormal{res,red}}}$, consisting of the restricted
    roots $\overline\alpha$ so that $2\overline\alpha$ is not a
    restricted root; equivalently, those falling in cases (1) and (2)
    of \eqref{eq:restypes}. This subsystem is preserved by $W^\theta$,
    and has the same Weyl group:
    $$W({\mathcal R}_{\text{\textnormal{res,red}}}) = W({\mathcal
      R}_{\text{\textnormal{res}}}).$$
      \item The {\em complex subsystem}, written ${\mathcal
    R}_{\text{\textnormal{res,cplx}}}$, consisting of the restrictions
    to $T_{c,0}$ of the complex roots and the corresponding
    coroots. This subsystem is preserved by the action of $W^\theta$,
    and so defines a normal subgroup
    $$W_{\text{\textnormal{cplx}}} \lhd W^\theta.$$
    \item The {\em imaginary subsystem}, written ${\mathcal
    R}_{\text{\textnormal{res,imag}}}$, consisting of the restrictions
    to $T_{c,0}$ of the imaginary roots and the corresponding
    coroots. This subsystem is preserved by the action of
    $W^\theta$, and so defines a normal subgroup
   $$W_{\text{\textnormal{imag}}} \lhd W^\theta.$$
    The imaginary roots have a ${\mathbb Z}/2{\mathbb Z}$
    grading in which the compact imaginary roots are even and the
    noncompact imaginary roots are odd (cf. \eqref{eq:imag}). This
    grading is respected by $W_{\text{\textnormal{cplx}}}$, but not
    usually by $W_{\text{\textnormal{imag}}}$.
 \item The {\em root datum for $K$}, written ${\mathcal R}_K$. Its roots are
   the disjoint union of the complex roots and the compact imaginary roots:
$${\mathcal R}_K = {\mathcal R}_{\text{\textnormal{cplx}}}
   \amalg {\mathcal R}_{\text{\textnormal{imag,cpt}}}.$$
\end{enumerate}
\end{proposition}
\begin{proof_sketch}
That the restricted roots are a root system is classical. The term
``restricted roots'' most often refers to restriction to the split
part of a maximally split torus. The fact those restrictions
constitute a root system is proved in \cite{Helgason}*{Section
  VII.2}. Helgason's arguments can be applied (with substantial
simplifications) to show that ${\mathcal R}_{\text{\textnormal{res}}}$
is a root datum.

Another classical fact is that ${\mathfrak t}_c$ contains regular
elements, so that no element of $W$ can fix all elements of
${\mathfrak t}_c$. This proves that restriction to $T_{c,0}$ is an
injective group homomorphism on $W^\theta$. That the image contains
$W({\mathcal R}_{\text{\textnormal{res}}})$ follows from
$$s_{\overline\alpha} = \begin{cases} s_\alpha|_{T_{c,0}} &
  \text{$\alpha$ imaginary}\\
  (s_\alpha s_{\theta\alpha})|_{T_{c,0}} & \text{$\alpha$ complex,
    $\alpha+\theta\alpha$ not a root}\\
  s_{\alpha+\theta\alpha}|_{T_{c,0}} & \text{$\alpha +\theta\alpha$ a
    root.}
  \end{cases}$$
  (Only the second assertion requires thought, and it is very easy.)

That $W^\theta$ is generated by elements of these three kinds is due
perhaps to Knapp; a proof may be found in \cite{IC4}*{Proposition
  3.12}.

For (1), the complex roots are precisely those having a non-trivial
restriction to the $-1$ eigenspace ${\mathfrak a}$ of $\theta$. That
$W^\theta$ preserves these roots is obvious. In particular, the
reflections in complex restricted roots preserve complex roots. This
last fact is the main part of the proof that the complex roots are a
root datum.

Part (2) is exactly parallel, except that this time the condition is
trivial restriction to ${\mathfrak a}$. The grading was already
explained after \eqref{eq:imag}; that it is preserved by compact
imaginary reflections is clear. We postpone for a moment the
assertion that it is preserved by complex reflections.

Because $K_0$ is a reductive algebraic group with maximal torus
$T_{c,0}$, we have the root datum of $K_0$ in $X^*(T_{c,0})$ and
$X_*(T_{c,0})$. Evidently this includes the compact imaginary
roots. From each complex root $\beta$ with root vector $X_\beta$ we
get a root vector
$$X_\beta + \theta X_\beta \in {\mathfrak k}$$
for $\overline \beta$; so the complex roots are automatically roots
for $K_0$. This proves (3).

Because the complex root reflections have representatives in $K$, they
must preserve the compact/noncompact grading on the imaginary
roots. This completes the proof of (2). \end{proof_sketch}

Recall that we have fixed a $\theta$-stable set of positive roots
$R^+_c$; this defines automatically a positive root system
$R^+_{\text{\textnormal{res}}}$ for the restricted roots, and also for the
complex, imaginary, and compact roots. Write $\Gamma_c$ for the Dynkin
diagram of $R^+_c$ (a graph with a vertex for each simple $\alpha$ and
an edge labelled $r$ from $\alpha$ to $\beta$ whenever $\alpha+r\beta$
is a root). Then $\theta$ defines an automorphism of $\Gamma_c$. The
Dynkin diagram $\Gamma_{\text{\textnormal{res}}}$ for the restricted
roots has as vertex set the orbits of $\theta$ on $\Gamma_c$. A fixed
point on $\Gamma_c$ corresponds to an imaginary simple root in
$\Gamma_{\text{\textnormal{res}}}$; an orbit consisting of two
non-adjacent simple roots $\alpha$ and $\theta(\alpha)$ corresponds to
a complex simple root $\overline\alpha$ in the second case of
\eqref{eq:restypes}; and an orbit consisting of two
adjacent simple roots $\alpha$ and $\theta(\alpha)$ corresponds to a
non-reduced complex simple root $\overline\alpha$ in the third case of
\eqref{eq:restypes}. (Such a vertex $\overline\alpha$ is joined to itself
in the restricted Dynkin diagram $\Gamma_{\text{\textnormal{res}}}$
since $2\overline\alpha$ is the (imaginary) root
$\alpha+\theta(\alpha)$. See for example the top right diagram in
Table \ref{tab:A4red}.)

The reduced restricted roots are the restrictions of roots which
involve either both or neither of a pair $(\alpha,\theta\alpha)$ of
adjacent simple roots. Such roots in $R$ are themselves a
$\theta$-stable subsystem $R_{\text{\textnormal red}}$. The simple
roots of $R_{\text{\textnormal red}}$ are those of $R$, except that
each adjacent complex pair $(\alpha,\theta\alpha)$ is replaced by the
single imaginary simple root $\alpha+\theta\alpha$.

This process is illustrated for $SL(5,{\mathbb R})$ in Table
\ref{tab:A4red}. The Dynkin diagram in the upper left is for $R$,
showing the action of $\theta$ reversing the line. The diagram on the
upper right is for the restricted roots, obtained by folding the
diagram on the left in half. The diagram on the lower left eliminates
the complex roots for which $2\overline{\alpha}$ is a root, by
replacing the two middle roots by their sum. On the lower right are
the restricted reduced roots: the complex restricted root
$\overline\alpha$ has been replaced by an imaginary (restricted) root
$\alpha +\theta\alpha$.  In each diagram imaginary vertices are
indicated with a filled circle, and complex vertices with an empty circle.

\begin{table}
\setlength{\unitlength}{1cm}
\qquad \qquad \qquad \begin{tabular}{cccc}
$R$&
\begin{picture}(2.8,1.1)(.4,.5)
  \put(1.0,.6){\LARGE{$\curvearrowleftright$}}
  \put(1.2,.9){$\theta$}
\multiput(.6,.5)(.5,0){3}{\line(1,0){.3}}
\multiput(.5,.5)(.5,0){4}{\circle{.15}}
\end{picture}
& $R_{\text{\textnormal{res}}}$&
\begin{picture}(1.0,.6)(.4,.3)
\multiput(.2,.5)(.7,0){2}{\circle{.15}} 
\multiput(.3,.45)(0,.1){2}{\line(1,0){.5}} 
\put(.43,.38){\Large{$>$}}
\put(1.2,.4){\small\begin{rotate}{90}$\circlearrowleft$\end{rotate}}
\end{picture}
\\
$R_{\text{\textnormal{red}}}$&
\begin{picture}(2.8,1.1)(.4,.5)
  \put(1.04,.6){\LARGE{$\curvearrowleftright$}}
  \put(1.24,.9){$\theta$}
\multiput(.9,.5)(.5,0){2}{\line(1,0){.3}}
\multiput(.8,.5)(1.0,0){2}{\circle{.15}} 
\put(1.3,.5){\circle*{.15}} 
\end{picture}
& $R_{\text{\textnormal{res,red}}}$&
\begin{picture}(1.0,.6)(.2,.3)
  \put(.2,.5){\circle{.15}} 
  \put(1.0,.5){\circle*{.15}} 
\multiput(.3,.45)(0,.1){2}{\line(1,0){.6}} 
\put(.43,.38){\Large{$<$}}
\end{picture}

\end{tabular}
\caption{Restricted and reduced roots for $SL(5,R)$}
\label{tab:A4red} 
\end{table} 

We offer one more example, the restricted root system for the split
real form of $E_6$. In this case the Cartan involution $\theta$
interchanges the long legs of the Dynkin diagram. There are no
adjacent pairs $(\alpha,\theta\alpha)$, so the restricted root system
is already reduced. Its Dynkin diagram is obtained by folding together
the long legs of the $E_6$ diagram, obtaining a diagram of type
$F_4$. Again imaginary vertices are illustrated with a filled circle,
and complex vertices with an empty circle.
\begin{table}
\setlength{\unitlength}{1cm}
\qquad \qquad \qquad \quad \begin{tabular}{cccc}
  $R$ 
  &\quad\begin{picture}(2.5,1)(.4,.8)
\multiput(1.5,.8)(0,.5){2}{\circle*{.15}}
\multiput(.6,.8)(.5,0){4}{\line(1,0){.3}}
\multiput(.5,.8)(.5,0){2}{\circle{.15}}
\multiput(2,.8)(.5,0){2}{\circle{.15}}
\put(1.5,.9){\line(0,1){.3}}
\put(1.2,.45){\LARGE{$\curvearrowbotleftright$}}
\put(1.4,.15){$\theta$}
\end{picture}

  & \qquad $R_{\text{\textnormal{res}}}$ 
  & \quad \begin{picture}(.6,2.2)(.4,1.2)
      \multiput(.5,1.9)(0,.6){2}{\circle*{.15}}
      \multiput(.5,.7)(0,.6){2}{\circle{.15}}
      \put(.5,2){\line(0,1){.4}}
      \put(.5,.8){\line(0,1){.4}}
      \multiput(.47,1.4)(.06,0){2}{\line(0,1){.4}} 
      \put(.611,1.45){\large\begin{rotate}{90}$<$\end{rotate}}
\end{picture}\\[.7cm]
\end{tabular}
\caption{Restricted roots for the split real form of $E_6$}
\label{tab:E6res} 
\end{table} 

\begin{subequations}\label{se:semidirect}

Returning to general $G$, define
\begin{equation}\label{eq:rhos}\begin{aligned}
2\rho_{\text{\textnormal{cplx}}} &= \sum_{\alpha \in
  R^+_{\text{\textnormal{cplx}}}} \alpha \in X^*(T_{c,0})\\
2\rho_{\text{\textnormal{imag}}} &= \sum_{\beta\in
  R^+_{\text{\textnormal{imag}}}} \beta \in X^*(T_{c,0})\\
2\rho_K &= \sum_{\gamma\in R^+_K} \gamma \in X^*(T_{c,0}).
\end{aligned}\end{equation}
Define the {\em singular imaginary roots} by
\begin{equation}
  R^{\text{\textnormal{sing}}}_{\text{\textnormal{imag}}} = \{\delta \in
R_{\text{\textnormal{res}}} \mid \langle
2\rho_{\text{\textnormal{cplx}}},\delta^\vee \rangle = 0\},
\end{equation}
the {\em singular complex roots} by
\begin{equation}
  R^{\text{\textnormal{sing}}}_{\text{\textnormal{cplx}}} = \{\delta \in
R_{\text{\textnormal{res}}} \mid \langle
2\rho_{\text{\textnormal{imag}}},\delta^\vee \rangle = 0\},
\end{equation}
and the {\em singular noncompact roots} by
\begin{equation}
  R^{\text{\textnormal{sing}}}_{\text{\textnormal{ncpt}}} = \{\delta \in
R_{\text{\textnormal{res}}} \mid \langle
2\rho_K,\delta^\vee \rangle = 0\}.
\end{equation}
\end{subequations}

Using these root systems, we can begin to understand the restricted
Weyl group $W^\theta = W_{\text{\textnormal{res}}}$.
\begin{proposition}\label{prop:semidirect}
  We use the notation of Proposition \ref{prop:resroots}.
\begin{enumerate}
  \item The weight $2\rho_{\text{\textnormal{cplx}}}$ is dominant for
    $R^+_{\text{\textnormal{res}}}$ and regular for the complex
    roots. Therefore the singular imaginary roots form a Levi
    subsystem in $R_{\text{\textnormal{res}}}$, consisting entirely of
    imaginary roots. We get a semidirect product decomposition
 $$W^\theta = W_{\text{\textnormal{cplx}}} \rtimes
    W^{\text{\textnormal{sing}}}_{\text{\textnormal{imag}}}.$$
\item The weight $2\rho_{\text{\textnormal{imag}}}$ is dominant for
    $R^+_{\text{\textnormal{res}}}$ and regular for the imaginary
    roots. Therefore the singular complex roots form a Levi
    subsystem in $R_{\text{\textnormal{res}}}$, consisting entirely of
    complex roots. We get a semidirect product decomposition
 $$W^\theta = W^{\text{\textnormal{sing}}}_{\text{\textnormal{cplx}}} \ltimes
    W_{\text{\textnormal{imag}}} .$$
\end{enumerate}
For the last items, we modify our $\theta$-stable choice of positive
roots $R^+_c$ to a new choice
$$R^{+,K}_c \quad\text{making $2\rho_K$ dominant}.$$

\begin{enumerate}[resume]
\item The weight $2\rho_K$ is dominant for
  $R^{+,K}_{\text{\textnormal{res}}}$ and regular for the roots of
  $K$. Therefore the singular noncompact roots form a Levi subsystem
  with respect to $R^{+,K}_{\text{\textnormal{res}}}$ consisting
  entirely of noncompact imaginary roots
  $$\{\pm\beta_1,\ldots,\pm\beta_r\}.$$
  This root system is of type $A_1^r$, so has Weyl group
  $$W^{\text{\textnormal{sing}}}_{\text{\textnormal{ncpt}}} =
  ({\mathbb Z}/2{\mathbb Z})^r.$$
\end{enumerate}

For $1\le j\le r$, choose a root $SL(2)$
  $$\phi_{\beta_j} \colon SU(1,1) \rightarrow G({\mathbb R}), \qquad
\phi_{\beta_j}\left(\Ad\begin{pmatrix} i &
0\\ 0&-i\end{pmatrix}(g)\right) = \theta(\phi_{\beta_j}(g))$$
as in \eqref{eq:imag}. Define
$$\sigma_j = \phi_{\beta_j}\begin{pmatrix} 0 & 1\\-1 &
    0\end{pmatrix},$$
a representative in $N_{G}(H_c)$ for the simple
reflection $s_{\beta_j}$, and
$$m_j = \phi_{\beta_j}\begin{pmatrix} -1 & 0\\0 & -1 \end{pmatrix} \in
T_{c,0}({\mathbb R}) \subset K({\mathbb R}).$$
For $B \subset \{1,\ldots,r\}$, define
$$H_B = \sum_{i\in B} \beta_j^\vee \in X_*(H_c), \qquad \sigma_B =
\prod_{j\in B} \sigma_j$$
$$s_B = \prod_{j\in B} s_j \in
W^{\text{\textnormal{sing}}}_{\text{\textnormal{ncpt}}} \qquad m_B = \prod_{j\in B} m_j
=\exp(2\pi i H_B/2) = \sigma_B^2.$$
Then
$$\theta(\sigma_B) = \sigma_B^{-1} = m_B\sigma_B.$$

\begin{enumerate}[resume]
  \item The Weyl group element $s_B$ admits a representative in $K$ if
    and only if there is a coweight $\ell_B \in X_*(H_c)$ satisfying
    $$\ell_B + \theta(\ell_B) = H_B.$$
    In this case the representative may be taken to be
    $$\widetilde{s}_B = \exp(\pi i \ell_B)\sigma_B.$$
\item Define
  $$
  W^{\text{\textnormal{sing}}}_{\text{\textnormal{ncpt}}}(K) = \{s_B \in
({\mathbb Z}/2{\mathbb Z})^r \mid H_B \in (1+\theta)X_*(H_c)\} \qquad (B \subset
  \{1,\ldots,r\}).$$
  Then there is a semidirect product decomposition
  $$W_K = W_{K_0} \rtimes
  W^{\text{\textnormal{sing}}}_{\text{\textnormal{ncpt}}}(K),$$
  the second factor being an abelian group with every element of order
  one or two. 
\item Define
  $$K^\sharp = \{k\in K \mid \Ad_G(k) \in \Ad_G(K_0)\} = K_0T_c$$
  (cf. Corollary \ref{cor:fin}; $K^\sharp({\mathbb R}) = G({\mathbb
  R})^\sharp \cap K$). Then 
  $$G({\mathbb R})/G({\mathbb R})^\sharp \simeq K({\mathbb
  R})/K^\sharp({\mathbb R})\simeq K/ K^\sharp\simeq
  W^{\text{\textnormal{sing}}}_{\text{\textnormal{ncpt}}}(K).$$
\end{enumerate}
\end{proposition}
\begin{proof_sketch}
We recommend examining Table \ref{tab:resroots} to get a more
concrete picture of the constructions in the proposition.

\begin{subequations}\label{se:semidirectproof}
Part (2) is \cite{IC4}*{Proposition 3.12(c)}; part (1) can be proven in
exactly the same way. For (3), the dominance of $2\rho_K$ comes from
the choice of positive roots, and the regularity for $K$ is a general
fact about positive root sums in a root datum. This implies that the
singular noncompact roots are roots in a Levi factor for the
restricted root datum, and are all noncompact imaginary. In
particular, the sum of two distinct singular noncompact roots cannot
be a root; for if it were root, the grading would necessarily make it
even, and so compact, and therefore not singular.

The absence of root sums shows that the noncompact singular system
consists of orthogonal simple roots, and is therefore of type $A_1^r$.
The assertions before (4) all take place in $SU(1,1)^r$, where they
are easy computations.

For (4), any representative of $s_B$ is of the form
\begin{equation}\widetilde{s}_B = h\sigma_B, \qquad \text{some\ } h =
  \exp(i\pi \ell) \in H_c.
\end{equation}
Therefore
$$\theta\widetilde{s}_B = \theta(h)m_B\sigma_B,$$
and $\widetilde{s}_B$ belongs to $K$ if and only if
$$h\sigma_B = \theta(h)m_B\sigma_B, \qquad h\theta(h)^{-1} = m_B,$$
or equivalently
\begin{equation}
  \exp(i\pi(\ell - \theta\ell)) =  \exp(-i\pi H_B).
\end{equation}
The kernel of $\exp(2\pi i)$ on ${\mathfrak h}_c$ is $X_*(H_c)$, so the
conclusion is that there must be an element $\ell_B\in X_*(H_c)$
satisfying
$$(\ell - \theta\ell)/2 + H_B/2 = \ell_B.$$
Because $H_B$ is in the $+1$ eigenspace of $\theta$ and
$\ell-\theta(\ell)$  in the $-1$ eigenspace, this equation is
equivalent to two equations
\begin{equation}\label{eq:repinK} H_B = \ell_B+\theta(\ell_B),\quad (\ell -
  \theta\ell)/2 = (\ell_B - \theta(\ell_B))/2.\end{equation}
So the existence of $\widetilde{s}_B$ guarantees the existence of
$\ell_B$ as the Proposition requires. Conversely, given $\ell_B$ as in
the proposition, choosing $\ell = \ell_B$ makes \eqref{eq:repinK}
true, proving that
\begin{equation}
  \widetilde{s}_B = \exp(i\pi\ell_B)\sigma_B
\end{equation}
is a representative for $s_B$ in $K$.

For (5), suppose $w\in W_K$. Clearly $w(R^+_K)$ is another positive
root system for $R_K$, so there is a unique $w_1\in W(R_K) = W_{K_0}$
satisfying
\begin{equation}
  w(R^+_K) = w_1(R^+_K), \qquad w(2\rho_K) = w_1(2\rho_K).
\end{equation}
Therefore $w_2 = w_1^{-1}w$ fixes $2\rho_K$. By Chevalley's theorem,
$w_2$ is a product of reflections fixing $2\rho_K$; that is, $w_2 \in
W^{\text{\textnormal{sing}}}_{\text{\textnormal{ncpt}}}$. Now (5) follows.

Part (6) is elementary.
\end{subequations} 
\end{proof_sketch}

If $\theta$ acts trivially on the roots in $H_c$, then all roots are
imaginary, and there is not much content to Proposition
\ref{prop:semidirect}(1)--(2). If $\theta$ interchanges two simple
factors $R_L\simeq R_0$ and $R_R\simeq R_0$ of the root system, then
all the roots are complex, and $W^\theta$ is the diagonal copy of
$W(R_0)$. The remaining and most interesting (indecomposable)
possibility is that $\theta$ acts as a nontrivial automorphism of
order $2$ of a simple root system $R$.  There is up to isomorphism
exactly one such automorphism for the simple root systems of types
$A_n$ $(n\ge 2)$, $D_n$ $(n\ge 4)$, and $E_6$, and none for the other
simple systems. Table \ref{tab:resroots} lists the restricted root
systems in each case, and some of the other root systems described in
Proposition \ref{prop:resroots}. In each case the last two columns
give two semidirect product decompositions of
$W_{\text{\textnormal{res}}} = W^\theta$ from Proposition
\ref{prop:semidirect}.

\medskip
\begin{table}
\begin{small}
\begin{tabular}{l|l|ll|ll| ll}
  \Bstrut $R$ & $R_{\text{\textnormal{res}}}$ &
  $R_{\text{\textnormal{cplx}}}$ &  \hskip -.3cm $R_{\text{\textnormal{imag}}}$ &
  $R^{\text{\textnormal{sing}}}_{\text{\textnormal{cplx}}}$ &
  \hskip -.2cm $R^{\text{\textnormal{sing}}}_{\text{\textnormal{imag}}}$ &
  $W_{\text{\textnormal{cplx}}} \rtimes
  W^{\text{\textnormal{sing}}}_{\text{\textnormal{imag}}}$ &
    $W^{\text{\textnormal{sing}}}_{\text{\textnormal{cplx}}} \ltimes
  W_{\text{\textnormal{imag}}}$ \\[.4ex]\hline
  \Tstrut$A_{2n-1}$ & $C_n$ & $D_n$ & $A_1^n$ & $A_{n-1}$ & $A_1$ & $W(D_n)
    \rtimes \{\pm 1\}$ & $S_n \ltimes \{\pm 1\}^n$ \\
  $A_{2n}$ & $BC_n$ & $B_n$ & $A_1^n$ & $A_{n-1}$ & $\emptyset$ &
    $W(B_n)\rtimes 1$ & $S_n \ltimes \{\pm 1\}^n$ \\
  $D_{n+1}$ & $B_{n}$ & $A_1^{n}$ & $D_{n}$ & $A_1$ & $A_{n-1}$ &
    $\{\pm 1\}^{n} \rtimes S_{n}$ & \hskip -.35cm $\{\pm 1\} \ltimes W(D_{n})$ \\
  $E_6$ & $F_4$ & $D_4$ & $D_4$ & $A_2$ & $A_2$ & $W(D_4) \rtimes S_3$
    & \hskip .1cm $S_3 \ltimes W(D_4)$\\
\end{tabular}\end{small}
\caption{\textbf{Restricted root systems}}
\label{tab:resroots}
\end{table}

One can give a similarly exhaustive enumeration of the results
of Proposition \ref{prop:semidirect}(3--6), but the details are
substantially more complicated; so we will content ourselves with a
few examples. If the complex group $G$ is simply connected, then $X_*$
is the coroot lattice, which has as a basis the simple
coroots. Because the roots $\beta_i$ are simple, the equation in
Proposition \ref{prop:semidirect}(4) can have no solution unless $B$
is empty. That is (still for $G$ simply connected)
$$W^{\text{\textnormal{sing}}}_{\text{\textnormal{ncpt}}}(K) = 1,
\qquad K=K^\sharp = K_0T_c.$$
(In fact $K$ must be connected in this case.)

We get interesting departure from this behavior only when $X_*$
includes more than the coroots. Enlarging $X_*$ means passing to
central quotients of $G$; the most interesting case is for the adjoint
group. Here are some examples.

\begin{subequations}\label{se:PSp}
Suppose first that $G=PSp(2n,{\mathbb R})$, the projective symplectic
group. In this case
\begin{equation}\begin{aligned}
  X^* &= \{\lambda\in {\mathbb Z}^n \mid \sum_i \lambda_i \in 2{\mathbb
  Z}\} \quad R(G,H_c) = \{\pm 2e_i, (\pm e_i \pm e_j) \mid  i \ne j \}\\
  X_* &= \langle {\mathbb Z}^n,(1/2,\ldots,1/2)\rangle \qquad
R^\vee(G,H_c) = \{\pm e_i, (\pm e_i \pm e_j) \mid i \ne j \}
\end{aligned}\end{equation}
The action of $\theta$ on $H_c$ is trivial, so all the roots are
imaginary. The compact ones are
\begin{equation}
  R_{\text{\textnormal{cpt}}} = \{(e_i -e_j) \mid i\ne j\}, \quad
  2\rho_K = (n-1,n-3,\ldots,-n+1) \in X^*.
\end{equation}
We therefore calculate
\begin{equation}
  R^{\text{\textnormal{sing}}}_{\text{\textnormal{ncpt}}} =
\{\pm(e_1+e_n),\ldots ,(e_{[n/2]}+e_{n-[n/2]+1})\} \cup \{2e_{(n+1)/2}\};
\end{equation}
the last root is present only if $n$ is odd. The corresponding simple coroots
are
$$\{(e_1 + e_n),\ (e_2+ e_{n-1}),\ \ldots,\ (e_{[n/2]}+e_{n-[n/2]+1})\}\cup
\{e_{(n+1)/2}\},$$
again with the last term present only if $n$ is odd. The elements
$H_B$ have all coordinates $1$ or $0$, symmetrically
distributed. Since $\theta$ acts by the identity, $B$ contributes to
$W^{\text{\textnormal{sing}}}_{\text{\textnormal{ncpt}}}$ if and only
if $H_B$ is divisible by two in $X_*$; that is, if and only if
$$B=\emptyset \quad \text{or}\quad B=\{1,\ldots,r\}.$$
The nontrivial Weyl group element is
\begin{equation}
  w_B(t_1,\ldots,t_n) = (t_n^{-1},\ldots,t_1^{-1})
\end{equation}
(reverse order and invert all entries). (More precisely, that is the
Weyl group element in $Sp(2n)$, acting on the maximal torus $({\mathbb
  C}^\times)^n$. In our case that torus is divided by $\pm 1$.)
\end{subequations} 

\begin{subequations}\label{se:PSO}
Suppose next that $G=PSO(2n,2n)$, the projective special orthogonal group
(the split form of $D_{2n}$). In this case
\begin{equation}\begin{aligned}
  X^* &= \{\lambda\in {\mathbb Z}^{2n} \mid \sum_i \lambda_i \in 2{\mathbb
  Z}\} \quad R(G,H_c) = \{(\pm e_i \pm e_j) \mid i \ne j \}\\
  X_* &= \langle {\mathbb Z}^{2n},(1/2,\ldots,1/2)\rangle \quad
R^\vee(G,H_c) = \{(\pm e_i \pm e_j) \mid i \ne j\}
\end{aligned}\end{equation}
We will sometimes write a semicolon between the first $n$ and the last
$n$ coordinates of $X^*$ for clarity.
The action of $\theta$ on $H_c$ is trivial, so all the roots are
imaginary. The compact ones are
\begin{equation}\begin{aligned}
  R_{\text{\textnormal{cpt}}} &= \{(\pm e_p \pm e_q),\ (\pm e_{n+p}\pm
  e_{n+q}) \mid 1 \le p\ne q\le n\}, \\
  2\rho_K &= (n-1,n-2,\ldots,1,0;n-1,n-2,\ldots,1,0) \in X^*.
  \end{aligned}
\end{equation}
We therefore calculate
\begin{equation}
  R^{\text{\textnormal{sing}}}_{\text{\textnormal{ncpt}}} =
\{\pm(e_p-e_{n+p})\mid 1\le p \le n-1\} \cup \{(e_n \pm e_{2n})\};
\end{equation}
The corresponding simple coroots are the same.
The elements $H_B$ have coordinates $1 \le p \le n-1$ equal to $1$
or $0$, with the same value on coordinate $p+n$. The coordinates $n$
and $2n$ are either $(0,0)$ or $(1,\pm 1)$ or $(2,0)$. Since $\theta$
acts by the identity, $B$ contributes to
$W^{\text{\textnormal{sing}}}_{\text{\textnormal{ncpt}}}(K)$ if and only
if $H_B$ is divisible by two in $X_*$; that is, if and only if
$$H_{B_0} = 0, \qquad B_0 = \emptyset;$$
$$H_{B_{\pm}} = (1,\ldots,1;-1,\ldots,\mp 1), \quad B_{\pm} = \{(e_p-e_{n+p})| p
\le n-1\}\cup\{(e_n\mp e_{2n})\} $$
$$H_{B_2} = (0,\ldots,2;0,\ldots,0), \qquad B_2 =\{(e_n - e_{2n})\ ,\ (e_n
+ e_{2n})\}.$$
The three nontrivial Weyl group elements are
\begin{equation}\begin{aligned}
  w_{B_{\pm}}(s_1,\ldots,s_n;t_1,\cdots,t_n) &= (t_1,\ldots,t_n^{\pm
  1};s_1,\ldots,s_n^{\pm 1})\\
w_{B_2}(s_1,\ldots,s_n;t_1,\ldots,t_n) &=
  (s_1,\ldots,s_{n-1},s_n^{-1};t_1,\ldots,t_{n-1},
t_n^{-1}).\end{aligned}
\end{equation}
(More precisely, those are the Weyl group elements in $SO(4n)$, acting on
the maximal torus $({\mathbb C}^\times)^{2n}$. In our case that torus
is divided by $\pm 1$.)
Because $T_c=H_c$ is connected, the group $K^\sharp = K_0T_c$ is
connected. Therefore the group of connected components of $K$ is
$$K/K_0 = W^{\text{\textnormal{sing}}}_{\text{\textnormal{ncpt}}}(K) =
({\mathbb Z}/2{\mathbb Z})^2,$$
the Klein four-group.
\end{subequations} 

\begin{subequations}\label{se:Wupper1}
We are going to need to understand the cosets of $W_{K_0}$ in
$W^\theta$. We conclude this section with that. Define
\begin{equation}
  W^1 = \{w\in W^\theta\mid wR^+_{\text{\textnormal{res}}} \supset
  R^+_K\};
\end{equation}
equivalently, these are the restricted Weyl group elements making only
noncompact imaginary roots change sign. The reason these elements are
of interest is that they are natural coset representatives for
$W_{K_0}$ in $W^\theta$:
\begin{equation}\label{eq:Wupper1}
  W^\theta = W_{K_0}\cdot W^1, \qquad W^1 \simeq W_{K_0}\setminus
  W^\theta.
\end{equation}
\end{subequations} 

\begin{corollary}\label{cor:Wupper1} In the setting of \eqref{se:Wupper1},
  $$W^1 \subset
  W_{\text{\textnormal{imag}}}^{\text{\textnormal{sing}}}.$$
  More precisely,
  $$W^1 = \left\{ w\in
  W_{\text{\textnormal{imag}}}^{\text{\textnormal{sing}}} \mid
  wR^{+,\text{\textnormal{sing}}}_{\text{\textnormal{imag}}} \supset
  R^{+,\text{\textnormal{sing}}}_{\text{\textnormal{imag,cpt}}}\right\},$$
  $$W_{K_0}\setminus W^\theta \simeq
  W_{K_{\text{\textnormal{imag}},0}}\setminus
    W_{\text{\textnormal{imag}}}.$$
The groups on the right in the last formula come from the (maximal
cuspidal Levi) subgroup
$$L_{\text{\textnormal{imag}}} = G^{A_c}$$
corresponding to the imaginary roots of $H_c$.
\end{corollary}
This is immediate from Proposition \ref{prop:semidirect}(1).

\section{Restricted weights} 
\label{sec:hwtform}
\setcounter{equation}{0}
There are many useful classical facts about the set of weights of a
finite-di\-men\-sion\-al representation, like the fact that all weights are
in the convex hull of the extremal weights. In this section we first
formulate those facts for restricted weights with respect to a
maximally compact Cartan. Then we consider the behavior of invariant
Hermitian forms on the restricted extremal weight spaces.

\begin{subequations}\label{se:reswts}
Fix therefore a $\theta$-stable system of positive roots
$$R_c^+ \subset R(G,H_c),$$
and a strongly integral  $R^+_c$-dominant weight
\begin{equation}\label{eq:domwt}
  \gamma_c = (\lambda_c,\nu_c). 
\end{equation}
Write
\begin{equation}
  F(\gamma_c) = \text{(some) finite-dimensional irreducible,
    highest weight $\gamma_c$}
\end{equation}
as in Proposition \ref{prop:fin}(3). Eventually we will impose also
the requirement
\begin{equation}\label{eq:domherm}
  \text{$\nu_c$ is purely imaginary;}
\end{equation}
the requirement that $\nu_c$ be imaginary is the condition from
Proposition \ref{prop:herm} for the existence of a $G({\mathbb R})^\sharp$
  invariant Hermitian form on $F(\gamma_c)$.

Every continuous character of $H_c({\mathbb R})$ restricts to a
continuous character of $T_c({\mathbb R})$, which is in turn the
restriction of a unique algebraic character in $X^*(T_c)$. The {\em
  restricted weights} of the finite-dimensional representation
$F(\gamma_c)$ are the characters
\begin{equation}
  \{\overline\phi \in X^*(T_c) \mid \text{$\overline\phi = $
    restriction of character $\phi$ of $H_c({\mathbb R})$ in $F(\gamma_c)$.}\}
\end{equation}
It was more convenient to discuss the general theory of restricted
roots on the connected torus $T_{c,0}$, but it is more convenient to
discuss restricted weights on all of $T_c$. Passage back and forth is
facilitated by the fact
\begin{equation}
  X^*(T_c) \buildrel \text{\textnormal{res}} \over \longrightarrow
  X^*(T_{c,0}) \quad \text{is injective on restricted root lattice ${\mathbb
      Z}R_{\text{\textnormal{res}}}$;}
\end{equation}
the lattice means the lattice of $T_c$-weights of $S({\mathfrak
  g})$. Using this fact, we will freely replace any restricted root
$\overline\alpha \in X^*(T_{c,0})$ by its unique extension to $T_c$ as
a weight of ${\mathfrak g}$.
Define
\begin{equation}
  2\rho^\vee_{\text{\textnormal{res}}} = \sum_{\overline\alpha \in
    R^+_{\text{\textnormal{res,red}}}} \overline\alpha^\vee,
\end{equation}
the sum of the coroots for the positive reduced restricted roots. If
$\overline\phi\in X^*(T_c)$ is any character, then there is a unique character
$w\phi$ (for $w\in W^\theta$) with the property that $w\overline\phi$ is weakly
dominant for $R^+_{\text{\textnormal{res}}}$. We define the {\em
  restricted height of $\overline\phi$} by
  \begin{equation}
    \HT_{\text{\textnormal{res}}}(\overline\phi) =
    \langle w\overline\phi,2\rho^\vee_{\text{\textnormal{res}}}\rangle =
\langle w\phi,2\rho^\vee_{\text{\textnormal{res}}}\rangle
  \end{equation}
  a nonnegative integer. (The last pairing is independent of the
  choice of $\phi\in X^*(H_c)$ restricting to $\overline\phi$, because
  the restricted coroots are $\theta$-fixed.) Clearly
  \begin{equation}
    \HT_{\text{\textnormal{res}}}(\overline\phi) =
    \HT_{\text{\textnormal{res}}}(x\overline\phi) \qquad (x\in
    W^\theta).
  \end{equation}
\end{subequations} 

Here is the description we want of restricted weights.
\begin{proposition}\label{prop:resext}
Suppose we are in the setting of \eqref{se:reswts} so that in
particular $F(\gamma_c)$ is an irreducible finite-dimensional
representation of $G({\mathbb R})$ of highest weight
$$\gamma_c = (\lambda_c,\nu_c).$$

\begin{enumerate}
  \item The set of restricted weights (and their multiplicities) is
    invariant under the restricted Weyl group $W^\theta$.
  \item An $R^+_{\text{\textnormal{res}}}$-dominant
    restricted weight $\overline\phi$ is a restricted weight of
    $F(\gamma_c)$ if and only if
    $$\lambda_c = \overline\phi + \sum_{\overline\alpha \in
    R^+_{\text{\textnormal{res}}}}
    n_{\overline\alpha}\overline\alpha, \qquad (n_{\overline\alpha}
    \in {\mathbb N}).$$
In this case
    $$\HT_{\text{\textnormal{res}}}(\overline\phi) \le
    \HT_{\text{\textnormal{res}}}(\lambda_c),$$
    with equality if and only if $\overline\phi = \lambda_c$.
  \item Suppose a restricted weight
    $\overline\phi$ is a weight of $F(\gamma_c)$. Then
    $$\lambda_c = \overline\phi + \sum_{\overline\alpha \in
    R^+_{\text{\textnormal{res}}}}
    n_{\overline\alpha}\overline\alpha, \qquad (n_{\overline\alpha}
    \in {\mathbb N})$$
  and
    $$\HT_{\text{\textnormal{res}}}(\overline\phi) \le
    \HT_{\text{\textnormal{res}}}(\lambda_c),$$
    with equality if and only if
    $$\overline\phi = w\lambda_c, \qquad \text{some $w\in
      W^\theta$}.$$
    We call $\{w\lambda_c\mid w\in W^\theta\}$ the {\em
      restricted extremal weights} of $F(\gamma_c)$.
  \item The $R^+_K$-dominant restricted extremal weights are
    $$W^1\lambda_c,$$
    with $W^1$ as in Corollary \ref{cor:Wupper1}. Each such extremal
    weight is therefore uniquely of the form
    $$w\lambda_c = \lambda_c - \sum_{\beta\in
      R^{+,\text{\textnormal{sing}}}_{\text{\textnormal{imag}}}
      \text{\ simple}} n_\beta\beta,$$
    with notation as in \eqref{se:semidirect}.
  \end{enumerate}
\end{proposition} 
Part (1) is elementary. Part (2) is exactly parallel to a standard
fact about weights of finite-dimensional representations, and can be
proved in the same way. Then (3) follows from (1) and (2). Part (4)
follows from Corollary \ref{cor:Wupper1}. We omit the details.

\begin{corollary}\label{cor:extrsig} Suppose we are in the setting of
  \eqref{se:reswts}, and that \eqref{eq:domherm} also holds, so that
  $F(\gamma_c)$ admits a $G({\mathbb R})_0$-invariant Hermitian form
  $$\langle\cdot,\cdot\rangle_{F(\gamma_c)}.$$
  We normalize this form to be positive on the $\lambda_c$ restricted
  weight space.
  \begin{enumerate}
    \item The form $\langle\cdot,\cdot\rangle_{F(\gamma_c)}$ is
      nondegenerate on each (one-dimensional) restricted extremal
      weight space $w\lambda_c$, and so either positive or negative there. Write
      $$\epsilon_{F(\gamma_c)}(w) = \pm 1$$
      for this sign.
      \item The sign $\epsilon_{F(\gamma_c)}(w)$ is invariant under left
        multiplication by $W(K_0)$, and so is determined by its
        restriction to the coset representatives $W^1$ of Corollary
        \ref{cor:Wupper1}.
        \item Write the simple roots for the Levi subsystem
          $R^{\text{\textnormal{sing}}}_{\text{\textnormal{imag}}}$ as
          the disjoint union of compact and noncompact imaginary roots:
          $$
          \Pi^{+,\text{\textnormal{sing}}}_{\text{\textnormal{imag}}}
          =
          \Pi^{+,\text{\textnormal{sing}}}_{\text{\textnormal{imag,cpt}}}
          \ \amalg\
          \Pi^{+,\text{\textnormal{sing}}}_{\text{\textnormal{imag,ncpt}}}$$
          (notation as in Proposition \ref{prop:resext}(4)). For $w\in W^1$, we
          have
          $$\epsilon_{F(\gamma_c)}(w) = \prod_{\beta\in
              \Pi^{\text{\textnormal{sing}}}_{\text{\textnormal{imag,ncpt}}}}
          (-1)^{n_\beta}.$$
        \item  The form $\langle\cdot,\cdot\rangle_{F(\gamma_c)}$ is
          invariant by $K({\mathbb R})$ (and therefore by $G({\mathbb
            R})$) if and only if
          $$\epsilon_{F(\gamma_c)}(xw) = 1, \qquad \text{all $x\in
              W^{\text{\textnormal{sing}}}_{\text{\textnormal{ncpt}}}(K)$}$$
          (see Proposition \ref{prop:semidirect}(5)).
  \end{enumerate}
\end{corollary}
\begin{proof}
\begin{subequations}{se:extrsig}
Because all characters of the compact group $T_c({\mathbb R})$ are
Hermitian, the Hermitian pairing necessarily makes the distinct
restricted weight spaces orthogonal, and so (by nondegeneracy) defines
a nondegenerate form on each restricted weight space. Now (1) is
immediate. The form is preserved by $G({\mathbb R})_0 \supset
K({\mathbb R})_0$, and the Weyl group elements in $W(K_0)$ have
representatives in $K({\mathbb R})_0$. So (2) follows. Part (3) can be
proven by induction on the length of $w$. It is obvious if $w=1$; so
suppose $w\ne 1$, and choose a simple reflection $s_\beta$ so that
\begin{equation}\ell(s_\beta w) = \ell(w) -1.
\end{equation}
Because $w$ is in the Levi subgroup
$W^{\text{\textnormal{sing}}}_{\text{\textnormal{imag}}}$, the root
$\beta$ must be imaginary. Define
\begin{equation}\label{eq:mbeta}
  m = \langle s_\beta w\lambda,\beta^\vee\rangle.
\end{equation}
Then
\begin{equation}
  w\lambda = s_\beta w\lambda - m\beta.
\end{equation}
The proposed formula for $\epsilon(w)$ therefore satisfies
\begin{equation}\label{eq:desideratum}
  \epsilon(w) = \epsilon(s_\beta w) \cdot \begin{cases} 1 & \text{if
    $\beta$ is compact} \\
    (-1)^{m} & \text{if $\beta$ is noncompact.}\end{cases}
  \end{equation}
To complete the induction argument, we must show that $\epsilon$
actually satisfies \eqref{eq:desideratum}.  The $m+1$-dimensional space
\begin{equation}
  E(w,\beta) = \text{span of the weight spaces } \{w\lambda - j\beta
  \mid 0 \le j \le m\}
\end{equation}
is an irreducible representation of $SL(2)$, by means of the root $SL(2)$
$\phi_\beta$ (see \eqref{se:realtori}). If $\beta$ is compact, $E(w,\beta)$
a Hermitian representation of $SU(2)$, so the form is definite, and
$\epsilon(s_\beta w) = \epsilon(w)$, as required by
\eqref{eq:desideratum}.

If $\beta$ is noncompact, then $E(w,\beta)$ is an irreducible Hermitian
representation of $SU(1,1)$. For such a representation, calculation in
$SU(1,1)$ shows that the signature of the form alternates in $j$ on
the weights $w\lambda-j\beta$. Consequently
\begin{equation}
  \epsilon(w) = (-1)^m\epsilon(s_\beta w),
\end{equation}
again as required by \eqref{eq:desideratum}. This completes the
induction, and the proof of (3).

For (4), if the form is $K({\mathbb R})$-invariant, then it must be
definite on each of the irreducible representations of $K({\mathbb
  R})$ generated by an extremal weight. Because the elements of
$W^{\text{\textnormal{sing}}}_{\text{\textnormal{ncpt}}}(K)$ have
representatives in $K({\mathbb R})$ (Proposition
\ref{prop:semidirect}(5)), the invariance property in (4) follows.

We omit the proof of the converse, which we will not use.
\end{subequations} 
\end{proof}

\section{Dirac operator and signature calculation}
\label{sec:dirac}
\setcounter{equation}{0}
We have so far avoided introducing invariant bilinear forms on
${\mathfrak g}$, because the idea of root data teaches us to do
that. But now it is time to talk about Dirac operators, and there the
choice of forms appears to be critical and unavoidable. We begin by
introducing the forms and the corresponding Casimir operators. (The
Casimir operators will play the role of Laplacians, of which the Dirac
operator is a kind of square root.)

\begin{subequations}\label{se:formsCas}
We continue to work with our complex connected reductive algebraic
group $G$ which is defined over ${\mathbb R}$, and with a chosen
Cartan involution $\theta$ as in \eqref{se:realred}, so that we have
$${\mathfrak g}({\mathbb R}) = {\mathfrak k}({\mathbb R}) +
{\mathfrak s}({\mathbb R})$$
as in \eqref{eq:cartandecompA}. Fix a non-degenerate
$\Ad(G)$-invariant symmetric bilinear form
\begin{equation}
  B\colon {\mathfrak g} \times {\mathfrak g} \rightarrow {\mathbb C}
\end{equation}
We require also that $B$ is preserved by $\theta$, and that
\begin{equation}
  \text{$B$ is real negative definite on ${\mathfrak k}({\mathbb
      R})$, real positive definite on ${\mathfrak s}({\mathbb R})$.}
\end{equation}
If $G$ is semisimple, the Killing form meets these requirements; in
general they are easy to achieve. The properties are inherited by many
real and $\theta$-stable reductive subalgebras. For example, if
$H=T({\mathbb R})A$ is a maximal torus as in \eqref{eq:maxtorB}, then
\begin{equation}
  \text{$B$ is real negative definite on ${\mathfrak t}({\mathbb
      R})$, real positive definite on ${\mathfrak s}({\mathbb R})$.}
\end{equation}
In particular $B$ is nondegenerate on ${\mathfrak h}$, and so dualizes
to a Weyl group invariant symmetric bilinear form $B^*$ on ${\mathfrak
  h}^*$. Because the roots take imaginary values on ${\mathfrak
  t}({\mathbb R})$ and real values on ${\mathfrak a}$, we get
\begin{equation}
  \text{$B^*$ is positive definite on the root lattice.}
\end{equation}

The decomposition
\begin{equation}
  {\mathfrak g} = [{\mathfrak g},{\mathfrak g}] + {\mathfrak
    z}({\mathfrak g}) = {\mathfrak g}_{\text{\textnormal{ss}}} + {\mathfrak
    z}({\mathfrak g})\end{equation}
(the second summand being the center) is orthogonal for $B$. On each
maximal torus this gives
\begin{equation}
  {\mathfrak h} = {\mathfrak h}_{\text{\textnormal{ss}}} + {\mathfrak
    z}({\mathfrak g}), \qquad {\mathfrak h}_{\text{\textnormal{ss}}} =
  {\mathfrak h} \cap [{\mathfrak g},{\mathfrak g}];
\end{equation}
the first summand is the span of the coroots. Dualizing gives an
orthogonal decomposition
\begin{equation}
  {\mathfrak h}^* = {\mathfrak h}^*_{\text{\textnormal{ss}}} + {\mathfrak
    z}({\mathfrak g})^*,
\end{equation}
and the first summand is the span of the roots.

If $\{X_i\}$ is any basis of ${\mathfrak k}$, there is a unique {\em
  dual basis} $\{X^j\}$ defined by the requirements
\begin{equation}
  B(X_i,X^j) = \delta_{ij}.
\end{equation}
The {\em Casimir operator for $K$} (with respect to $B$) is
\begin{equation}
  \Omega_K = \sum_i X_i X^i \in U({\mathfrak k}).
\end{equation}
It is independent of the choice of basis, and is fixed by $\Ad(K)$; in
particular, it belongs to the center of the enveloping algebra
$U({\mathfrak k})$. Consequently $\Omega_K$ acts by a complex scalar
operator
\begin{equation}
\mu(\Omega_K) \in {\mathbb C}
\end{equation}
on any irreducible representation $\mu$ of ${\mathfrak k}$.
In the same way, if $\{Z_p\}$ is any basis of
${\mathfrak g}$ and $\{Z^q\}$ the dual basis, we get the {\em Casimir
  operator for $G$}
\begin{equation}
  \Omega_G = \sum_p Z_p Z^p \in U({\mathfrak g}),
\end{equation}
which acts by a complex scalar
\begin{equation}
\pi(\Omega_G) \in {\mathbb C}
\end{equation}
on any irreducible representation $\pi$ of ${\mathfrak g}$.

If $\mu$ is an irreducible representation of $K$ of highest weight
$\xi\in X^*(T_c)$ with respect to $R^+_K$ (see \eqref{eq:rhos}), then
\begin{equation}\begin{aligned}
  \mu(\Omega_K) &= B(d\xi+d(2\rho_K)/2,d\xi+d(2\rho_K)/2 \\ &\quad -
  B(d(2\rho_K)/2,d(2\rho_K)/2) \ge 0;\end{aligned}
\end{equation}
equality holds if and only if $d\xi$ vanishes on all coroots of
$K$. In accordance with our policy of ignoring the difference between
characters and their differentials when it is harmless, we will
usually write this result as
$$\mu(\Omega_K) = B(\xi+\rho_K,\xi+\rho_K) - B(\rho_K,\rho_K).$$

In the same way, if $(\pi,F(\gamma_c))$ is an irreducible representation of
$G$ as in \eqref{se:reswts}, then
\begin{equation}\begin{aligned}
    \pi(\Omega_G) &= B(\gamma_c+\rho,\gamma_c+\rho) - B(\rho,\rho)\\
    &= B(\lambda_c+\rho,\lambda_c+\rho) - B(\rho,\rho) +
    B(\nu_c,\nu_c).
    \end{aligned}
\end{equation}
\end{subequations} 

\begin{subequations}\label{se:clifford}
We turn now to the Dirac operator. The key to its definition is the
(positive definite) real quadratic space
\begin{equation}
  ({\mathfrak s}({\mathbb R}),B), \qquad \Ad\colon K \rightarrow
  O({\mathfrak s}({\mathbb R}),B).
\end{equation}
The {\em Clifford algebra $C({\mathfrak s}({\mathbb R}))$} is the real
associative algebra with $1$ generated by ${\mathfrak s}({\mathbb R})$
subject to the relations
\begin{equation}
  X^2 + B(X,X) = 0 \qquad (X\in {\mathfrak s}({\mathbb R})),
\end{equation}
or equivalently
\begin{equation}
  XY + YX + 2B(X,Y) = 0 \qquad (X,Y\in {\mathfrak s}({\mathbb R})).
\end{equation}
By definition $C({\mathfrak s}({\mathbb R}))$ is a quotient of the
tensor algebra of ${\mathfrak s}({\mathbb R})$, from which it inherits
a filtration indicated by lower subscripts:
$$C({\mathfrak s}({\mathbb R}))_m = \text{span of products of at most
  $m$ elements of ${\mathfrak s}({\mathbb R})$.} $$
We have
\begin{equation}
  \gr C({\mathfrak s}({\mathbb R})) \simeq {\textstyle\bigwedge} {\mathfrak
    s}({\mathbb R}).
\end{equation}
\end{subequations} 

Here are the basic facts about the spin cover of a compact
orthogonal group.
\begin{proposition} \label{prop:spingroup} Define
  $$C({\mathfrak s}({\mathbb R}))^\times = \text{invertible elements
    of the Clifford algebra,}$$
    an open subgroup of the algebra. The conjugation action of this
    group on the Clifford algebra is by algebra automorphisms.
Regard $C({\mathfrak s}({\mathbb R}))$ as a Lie algebra under the
commutator of the associative algebra structure; this is the Lie
algebra of the group $C({\mathfrak s}({\mathbb R}))^\times$. Then there is a
natural inclusion of Lie algebras
  $${\mathfrak s}{\mathfrak o}({\mathfrak s}({\mathbb R})) \simeq
  {\textstyle \bigwedge^2 {\mathfrak s}({\mathbb R})} \buildrel j \over
  \hookrightarrow C({\mathfrak s}({\mathbb R}))_2,$$
  The {\em spin group} 
    is by definition
the corresponding Lie subgroup of $C({\mathfrak s}({\mathbb
  R}))^\times$:
$$\Spin({\mathfrak s}({\mathbb R})) = \exp\big(j({\mathfrak s}{\mathfrak
  o}({\mathfrak s}({\mathbb R})))\big) \subset C({\mathfrak s}({\mathbb
  R}))^\times.$$
The spin group action on $C({\mathbb R})$ by conjugation preserves
the filtration, and so descends to an action on
$$ \gr C({\mathfrak s}({\mathbb R})) \simeq {\textstyle\bigwedge} {\mathfrak
  s}({\mathbb R}),$$
The action on ${\mathfrak s}({\mathbb R})$ preserves the quadratic
form (because it comes from Clifford algebra automorphisms), so defines
$$\Spin({\mathfrak s}({\mathbb R})) \buildrel\pi\over\longrightarrow
SO({\mathfrak s}({\mathbb R})).$$
The differential of $\pi$ is the inverse of the Lie algebra
isomorphism $j$; so $\pi$ is a covering map. As long as
$\dim{\mathfrak s}({\mathbb R}) \ge 2$, we have
$$\ker \pi = \{\pm 1\} \subset C({\mathfrak s}({\mathbb R}))^\times,$$
so the covering is two to one.
\end{proposition}

Here is the representation theory of the Clifford algebra.
\begin{proposition} \label{prop:spinrep} Write
  $$\dim{\mathfrak s}({\mathbb R})=_{\text{\textnormal{def}}} n =
  2m+\epsilon, \qquad m=[n/2].$$
  The complexified Clifford algebra has dimension $2^n =
  2^\epsilon\cdot (2^m)^2.$ It is the direct sum of $2^\epsilon =
  1\text{\ or } 2$ copies of a matrix algebra of rank $2^m$. In
  particular, the center of the Clifford algebra has dimension
  $2^\epsilon$; it is spanned by $1$ and (if $n$ is odd)
  $$z=e_1\cdots e_{2m+1},$$
  with $\{e_i\}$ an orthonormal basis of ${\mathfrak s}({\mathbb
    R})$. This central element depends only on the orientation
  defined by the chosen orthonormal basis, and satisfies
  $$z^2 = (-1)^{m-1}.$$

  The Clifford algebra has $2^\epsilon$ irreducible representations,
  called {\em spin representations}, each of dimension $2^m$. In case
  $n$ is odd, these two representations are distinguished by the
  scalar by which $z$ acts: we write $(\sigma_{[\pm]},S_{[\pm]})$ for
  an irreducible representation on which $z$ acts by $\pm i^{m-1}$.

  If $n$ is even, the spin representation $(\sigma,S)$ has a ${\mathbb
    Z}/2{\mathbb Z}$ grading
  $$S = S_+ \oplus S_-,$$
  with each summand of dimension $2^{m-1}$. The generators $X\in
  {\mathfrak s}({\mathbb R})$ carry $S_+$ to $S_-$. The action of the
  spin group
$$\Spin({\mathfrak s}({\mathbb R})) \subset C({\mathfrak s}({\mathbb
    R}))^\times$$
  preserves $S_{\pm}$, and acts irreducibly on each; these are the
  {\em half-spin representations $\sigma_{\pm}$} of (the double
  cover of) an even special orthogonal group.

  If $n$ is odd, the two spin representations
  $\left(\sigma_{[\pm]},S_{[\pm]}\right)$ are
  isomorphic as representations of the spin group. The action is
  irreducible; this is the {\em spin representation $\sigma$} of (the double
  cover of) an odd special orthogonal group.

  Suppose that the weights for $SO({\mathfrak s}({\mathbb R}))$ acting
  ${\mathfrak s}({\mathbb C})$ are
  $$\{\pm\mu_1,\ldots,\pm \mu_m\}\ \amalg \ \{0\};$$
  the last zero is present only if $\epsilon =1$. Then the weights of
  (either) spin representation $S$ are
  $$(1/2)\sum_{j=1}^m \epsilon_j\mu_j,$$
  with $\epsilon_j= \pm 1$. Each such weight has multiplicity one.
\end{proposition}

It is possible to enlarge $\Spin({\mathfrak s}({\mathbb R})) \subset
C({\mathfrak s}({\mathbb R}))^\times$ to a double cover of the full
orthogonal group $O({\mathfrak s}({\mathbb R}))$. This is interesting
for us because
$$\Ad\colon K \rightarrow O({\mathfrak s}({\mathbb R}))$$
need not have image inside $SO$. All of the discussion starting in
\eqref{se:dirac} below can accordingly be extended to some double
cover $\widetilde K$ of $K$. But this is a bit complicated, and plays
no essential role in this paper; so we omit it.

\begin{subequations}\label{se:cliffherm}
  The real form $C({\mathfrak s}({\mathbb R}))$ of
  the complexified Clifford algebra corresponds to a conjugate-linear
  automorphism
  \begin{equation}
    \sigma_{\mathbb R}\colon C({\mathfrak s}({\mathbb R}))_{\mathbb C}
    \rightarrow C({\mathfrak s}({\mathbb R}))_{\mathbb C}, \qquad
    \sigma_{\mathbb R}(X) = X \qquad (X\in {\mathfrak s}({\mathbb R})).
    \end{equation}
  There is also a (complex-linear) algebra antiautomorphism $\tau$
  characterized by
  \begin{equation}\label{eq:cliffanttau}
    \tau(X) = -X \qquad (X\in {\mathfrak s}({\mathbb R})).
  \end{equation}
(The reason for the existence of $\tau$ is that the requirement
  \eqref{eq:cliffanttau} respects the defining relations of the
  Clifford algebra.) If $(\pi,M)$ is any $C({\mathfrak s}({\mathbb
    R}))_{\mathbb C}$-module, the Hermitian dual vector
  space $M^h$ (consisting of conjugate-linear functionals on $M$; see
  for example \cite{AvTV}*{Section 8}) becomes a $C({\mathfrak s}({\mathbb
    R}))_{\mathbb C}$-module by the requirement
\begin{equation}\label{eq:cliffhermA}
  \pi^h(c) = \pi\big(\tau(\sigma_{\mathbb R}(c))\big)^h \qquad (c \in
  C({\mathfrak s}({\mathbb R}))_{\mathbb C})
\end{equation}
or equivalently
\begin{equation}\label{eq:cliffhermB}
  \langle m, X\cdot\mu\rangle = \langle -X\cdot m,\mu\rangle \qquad
  (m\in M, \mu \in M^h, X\in {\mathfrak s}({\mathbb R})).
\end{equation}
Here we write $\langle\cdot,\cdot\rangle$ for the Hermitian pairing
between $M$ and its Hermitian dual $M^h$. Passage to the Hermitian
dual obviously fixes the unique simple $C({\mathfrak s}({\mathbb
  R}))_{\mathbb C}$-module $S$ in the even-dimensional case, so $S$
admits an invariant Hermitian form
\begin{equation}
  \langle\cdot,\cdot\rangle_S\colon S \times S \rightarrow {\mathbb
    C}.
\end{equation}

In the odd-dimensional case, we find for the central element $z$
described in Proposition \ref{prop:spinrep} that
$$\sigma_{\mathbb R}(z) = z, \qquad \tau(z) = (-1)^{m-1}z$$
Since $z$ acts on $S_\pm$ by the scalar $(\pm i)^{m-1}$, it follows that $z$
acts on the Hermitian dual $S_\pm^h$ by the scalar
$$(-1)^{m-1} \overline{(\pm i)^{m-1}} = (\pm i)^{m-1}.$$
Therefore $S_\pm^h \simeq S_\pm$, and $S_\pm$ admits an invariant
Hermitian form
\begin{equation}
  \langle\cdot,\cdot\rangle_{S_{\pm}}\colon S_\pm \times S_\pm
  \rightarrow {\mathbb C}.
\end{equation}
\end{subequations} 

\begin{proposition}\label{prop:cliffherm}
  In the setting of \eqref{se:cliffherm}, the invariant Hermitian
  forms $\langle,\rangle_S$ and $\langle,\rangle_{S_\pm}$ are all
  definite. We normalize them henceforth to be {\em positive}. The
  characteristic invariance property is
  $$\langle X\cdot s,s'\rangle + \langle s,X\cdot s'\rangle = 0 \qquad
  (X\in {\mathfrak s}({\mathbb R}));$$
  that is, the action of Clifford multiplication is by skew-adjoint
  operators.

  These Hermitian forms are also invariant under the action of the
  spin group $\Spin({\mathfrak s}({\mathbb R}))$.
\end{proposition}

\begin{subequations} \label{se:dirac}
Suppose
  \begin{equation}
    \text{$(\xi,V)$ is a $({\mathfrak g},K_0)$-module;}
  \end{equation}
that is, that $V$ is at the same time a complex representation of the
Lie algebra ${\mathfrak g}$, and a locally finite continuous
representation of the Lie group $K$, and that
\begin{equation}
  \text{the differential of $\xi|_K$ is equal to the restriction to
    ${\mathfrak k}_0$ of $\xi|_{\mathfrak g}$.}
\end{equation}
Let $(\sigma,S)$ be a spin representation of the complexified Clifford algebra
$C({\mathfrak s}_{\mathbb R})_{\mathbb C}$. In the odd-dimensional
case, we simply choose one of the two representations $S_+$ or
$S_-$. 
Finally, fix any basis
\begin{equation}
  \{X_1,\cdots,X_n\} \subset {\mathfrak s}({\mathbb R})
\end{equation}
for the $-1$ eigenspace of the Cartan involution on the real Lie
algebra, and let
\begin{equation}
  \{X^1,\cdots,X^n\} \subset {\mathfrak s}({\mathbb R}), \qquad
  B(X_i,X^j) = \delta_{ij}
\end{equation}
be the dual basis with respect to the symmetric invariant form $B$ of
\eqref{se:formsCas}. The {\em Dirac operator for $(\xi,V)$} is the
linear operator on $V\otimes S$ defined by
\begin{equation}\label{eq:D}
  D= \sum_{j=1}^n \xi(X_j)\otimes \sigma(X^j) \in \End(V\otimes S).
\end{equation}
It will be convenient as in the discussion of the Clifford algebra to
write
\begin{equation}
  n=2m+\epsilon, \qquad \dim S = 2^m = 2^{[n/2]}.
\end{equation}

The adjoint action defines a group homomorphism
\begin{equation}
  \Ad\colon K_0 \rightarrow SO({\mathfrak s}({\mathbb R})).
\end{equation}
Using the covering
\begin{equation}
 1\longrightarrow \{\pm 1\} \longrightarrow \Spin({\mathfrak
   s}({\mathbb R})) \buildrel\pi\over\longrightarrow SO({\mathfrak
   s}({\mathbb R})) \longrightarrow 1,
\end{equation}
from Proposition \ref{prop:spingroup}, we can define a pushout
\begin{equation}
  \widetilde K_0 = \{(s,k) \in \Spin({\mathfrak
    s}({\mathbb R})) \times K_0 \mid \Ad(k) = \pi(s)\}.
\end{equation}
There is a short exact sequence
\begin{equation}
  1\longrightarrow \{\pm 1\} \longrightarrow \widetilde K_0
  \buildrel \pi \over\longrightarrow K_0 \longrightarrow 1.
  \end{equation}
Projection on the first factor defines a homomorphism
$\widetilde\Ad$,
\begin{equation}
\widetilde\Ad\colon \widetilde K_0 \longrightarrow \Spin({\mathfrak
  s}({\mathbb R})).
\end{equation}
In this way $S$ becomes a representation of $\widetilde K_0$ by
\begin{equation}
  \sigma_{\widetilde K_0} = \sigma\circ\widetilde \Ad.
\end{equation}

The nonzero weights of $T_c$ on ${\mathfrak s}$ are
\begin{equation}\label{eq:ncptroots}
  \{\pm\gamma_j \mid 1\le j \le r\} \ \amalg \ \{0\}.
\end{equation}
Here the $\gamma_j$ are the complex positive roots and the noncompact
imaginary positive roots; and the multiplicity of the weight zero is
$\dim A_c$. In light of Proposition \ref{prop:spinrep}, it follows
that the weights of $\widetilde K_0$ on $S$ are
\begin{equation}\label{eq:spinwts}
  (1/2)\sum_{j=1}^r \epsilon_j\gamma_j,
\end{equation}
with $\epsilon_j = \pm 1$. The multiplicity of such a weight is the
number of expressions for it of this form, times $2^{[\dim
    A_c]/2}$. (The multiplicity arises because the weights $\mu_j$
appearing in Proposition \ref{prop:spinrep} are the $r$ pairs $\pm\gamma_j$,
together with $[\dim A_c/2]$ pairs of zeros.)

Of course $V$ is a representation of $\widetilde K_0$ by
$\xi\circ\pi$, which we will just call $\xi$.  Therefore
\begin{equation}
  (\sigma_{\widetilde K_0} \otimes\xi,S\otimes V)
\end{equation}
is a representation of $\widetilde K_0$.
\end{subequations} 

Here are the basic facts about Parthasarthy's Dirac operator.
\begin{proposition}[Parthasarathy \cite{Pdirac}]\label{prop:dirac} In
  the setting of \eqref{se:dirac}, the Dirac operator $D$ is independent of the
  choice of basis (of ${\mathfrak s}({\mathbb R})$), and commutes with
  the representation $\sigma_{\widetilde K_0} \otimes\xi$ of
  $\widetilde K_0$. Consequently
  $$\ker D \subset S\otimes V$$
  is a representation of $\widetilde K_0$ (as indeed is every
  eigenspace of $D$).

  The square of the Dirac operator is
  $$D^2 = - 1_S\otimes\xi(\Omega_G) + (\sigma_{\widetilde K_0}\otimes
  \xi)(\Omega_K) - [B(\rho_G,\rho_G) - B(\rho_K,\rho_K)]\cdot
  1_S\otimes 1_V.$$

  Suppose next that $\Omega_G$ acts on $V$ by a complex scalar
  $\xi(\Omega_G)$ (as is automatic if $\xi$ is irreducible). Then
  $D^2$ is diagonalized by the decomposition of $\sigma_{\widetilde
    K_0}\otimes \xi$ into irreducible representations of $\widetilde
  K_0$. All of the eigenvalues differ from $\xi(\Omega_G)$ by real
  scalars.

  Suppose finally that $V$ admits a nondegenerate invariant Hermitian
  form $\langle,\rangle_V$. Then $D$ is self-adjoint for the Hermitian
  form
  $$\langle,\rangle_S \otimes \langle,\rangle_V.$$
  If $V$ has signature $(p,q)$, then $S\otimes V$ has signature
  $(2^mp,2^mq)$ (notation as in Proposition \ref{prop:spinrep}).
\end{proposition}

Here is Kostant's result about the spectrum of the Dirac operator.

\begin{proposition}[\cite{Kmult}]\label{prop:Kmult} Suppose that
  $F(\gamma_c)$ is an irreducible finite-dimensional representation of
  $G({\mathbb R})$ of highest weight
  $$\gamma_c =  (\lambda_c,\nu_c)$$
  as in Proposition \ref{prop:resext}, and $S$ is a spin
  representation of $\Spin({\mathfrak s}({\mathbb R}))$ as in
  Proposition \ref{prop:spinrep}. Regard $S\otimes F(\gamma_c)$ as a
  representation of $\widetilde K_0$ as in Proposition
  \ref{prop:dirac}.
  \begin{enumerate}
\item Every irreducible representation $\widetilde\tau$ of $\widetilde
  K_0$ on $S\otimes F(\gamma_c)$ has highest weight of the form
  $$\overline\phi + w\rho_G - \rho_K - 2\rho(B),$$
  for some $w\in W^1$ (see Corollary \ref{cor:Wupper1}),
  $\overline\phi$ a $wR^+_{\text{\textnormal{res}}}$-dominant
  restricted weight of $F(\gamma_c)$, and $B$ a set of noncompact
  imaginary roots in $wR^+$.
\item The scalar $\widetilde{\tau}(\Omega_K)$ satisfies
  $$\widetilde\tau(\Omega_K) \le \langle\lambda_c+\rho_G,\lambda_c+\rho_G\rangle
  -\langle \rho_K,\rho_K\rangle.$$
  Equality holds if and only if
  $$\overline\phi = w\lambda_c = w'\lambda_c, w\rho_G - 2\rho(B) =
  w'\rho_G$$
  for some $w'\in W^1$. In particular, this largest possible
  eigenvalue of $\widetilde\tau(\Omega_K)$ is equal to
  $$\langle \lambda_c+ \rho_G,\lambda_c+\rho_G\rangle -
  \langle\rho_K,\rho_K\rangle.$$
  \end{enumerate}
\end{proposition} 

The proposition has been formulated in such a way as to outline its
proof in \cite{Kmult} and \cite{HKP}. The highest weight of any
representation of $K_0$ in $F(\gamma_c)$ must be a $K_0$-dominant
restricted weight of $F(\gamma_c)$, and therefore a
$wR^+_{\text{\textnormal{res}}}$-dominant restricted weight
$\overline\phi$. The highest weight of $\widetilde\tau$ must therefore
be equal to such a weight, plus a weight of $S$. A weight of $S$ is of
the form $w\rho_G - \rho_K - 2\rho(B)$. This is how (1) is proved. Now
the formula in \eqref{se:formsCas} for the eigenvalue of $\Omega_K$,
together with Proposition \ref{prop:resext}(2), leads easily to
(2).

\begin{corollary}\label{cor:D2} Suppose we are in the setting of Proposition
  \ref{prop:Kmult}.
  \begin{enumerate}
  \item The eigenvalues of $D^2$ on $S\otimes F(\gamma_c)$ are less than or equal
  to the positive number
  $$-\langle \nu_c,\nu_c\rangle.$$
  Equality occurs exactly on the representations of $\widetilde K_0$
  of highest weights
  $$w(\lambda_c + \rho_G) -\rho_K \qquad (w\in W^1),$$
  with $W^1$ as in Corollary \ref{cor:Wupper1}.
  \item Each such representation of $\widetilde K_0$ has multiplicity
    $$2^{[\ell/2]}, \qquad \ell = \dim A_c.$$
    \item The Hermitian form on $S\otimes F(\gamma_c)$ is definite on
      each such representation, of sign $\epsilon(w)$ computed in
      Corollary \ref{cor:extrsig}(4).
    \item Define
      $$p_0 = \sum_{w\in W^1,\epsilon(w)=+1} \dim
      E(w(\lambda_c+\rho_G) - \rho_K),$$
      $$q_0 = \sum_{w\in W^1,\epsilon(w)=-1} \dim
      E(w(\lambda_c+\rho_G) - \rho_K),$$
Then the the signature of the form on the largest eigenspace of
      $D^2$ is
      $$2^{[\ell/2]}(p_0,q_0).$$
  \end{enumerate}
\end{corollary}

\begin{proof_sketch}
Part (1) is precisely Proposition \ref{prop:Kmult}, together with
Par\-tha\-sa\-rathy's formula in Proposition \ref{prop:dirac} for
$D^2$. For (2), the proof of Proposition \ref{prop:Kmult} shows that
the multiplicity of such a representation of $\widetilde K_0$ is equal to the
multiplicity of a highest weight space of $S$ of weight $\rho_G -
\rho_K$. (That was the reason for recalling the proof above.) This
last weight multiplicity is computed after \eqref{eq:spinwts}; it is 
$2^{[\ell/2]}$. For (3), this same proof shows that the highest weight
vector of such a representation is equal to a weight vector in
$F(\gamma_c)$ of weight $w\lambda_c$ (which by definition has length a
positive multiple of $\epsilon(w)$) tensored with a vector in $S$
(which has positive length). Part (4) just writes (3) explicitly.
\end{proof_sketch}

\begin{lemma}\label{lemma:selfadjsig} Suppose that $T$ is a linear
  operator on a finite-dimensional Hermitian vector space $V$,
  self-adjoint with respect to a Hermitian form of signature
  $(P,Q)$; and suppose that $T$ has purely imaginary eigenvalues.
  \begin{enumerate}
  \item For $x\ne 0$, the Hermitian form defines an isomorphism
    $$V_{ix}^h \simeq V_{-ix}.$$
    In particular, the eigenspaces $V_{ix}$ and $V_{-ix}$ have the
    same dimension $m(x)$, and contribute $(m(x),m(x))$ to the
    signature.
    \item The Hermitian form has a nondegenerate restriction to the
      kernel
      $$V_0 = \ker T,$$
      where it has signature $(p_1,q_1)$.
    \item The signatures on $V$ and $V_0$ satisfy
$$P-Q = p_1 - q_1.$$
In particular, the Signature invariant for $V$ is equal to that for
$V_0$:
        $$ \Sig(V) = |P-Q| = |p_1-q_1| = \Sig(V_0).$$
  \end{enumerate}
\end{lemma}

Once stated, this result is immediate; what is true is
$$P = p_1 + \sum_{x>0} m(x), \qquad Q = q_1 + \sum_{x>0} m(x).$$

Here at last is the main theorem.

\begin{theorem}\label{thm:Gsig} Suppose in the setting of
  \eqref{se:reswts} and \eqref{eq:domherm} that $F(\gamma_c)$ is a
  finite-dimensional representation of $G({\mathbb R})$ admitting an
  invariant Hermitian form $\langle\cdot,\cdot\rangle_{F(\gamma_c)}$;
  we normalize the form to be positive on the $\gamma_c$ weight
  space. Write $2r$ for the number of noncompact imaginary and complex
  restricted roots of $T_c$ in $G$:
$$2r = \dim G/H_c - \dim K/T_c.$$
Then (with notation as in Corollary \ref{cor:D2})
$$\Sig(F(\gamma_c)) = |p_0 - q_0|/2^r.$$
\end{theorem}

\begin{proof}
  \begin{subequations}\label{se:Gsig}
The approximate idea is to apply Lemma \ref{lemma:selfadjsig} to the
Dirac operator $D$. This is indeed a self-adjoint linear operator on the
finite-dimensional Hermitian vector space $F(\gamma_c)\otimes
S$. Write $(p,q)$ for the signature of the form on $F(\gamma_c)$; then
the form on $F(\gamma_c)\otimes S$ has signature
\begin{equation}\label{eq:bigsig}
  (P,Q) = 2^m(p,q) \qquad (m = [\dim {\mathfrak s}/2])
\end{equation}
(see Proposition \ref{prop:spinrep}). Corollary \ref{cor:D2} says that
the eigenvalues of $D^2$ are less than or equal to $-\langle
\nu_c,\nu_c\rangle$, and that the signature of the form on the largest
eigenspace is
\begin{equation}
  2^{[\ell/2]}(p_0,q_0).
\end{equation}

Suppose for a moment that
\begin{equation}\label{eq:kerD}
  \nu_c=0.
\end{equation}
Then the Corollary says that the
eigenvalues of $D^2$ are less than or equal to zero. From this it
follows that the eigenvalues of $D$ (as square roots of non-positive
real numbers) are purely imaginary. Therefore Lemma
\ref{lemma:selfadjsig} applies, and tells us that
\begin{equation}
  P-Q = 2^{[\ell/2]}(p_0-q_0).
\end{equation}
Combining this with \eqref{eq:bigsig} gives
\begin{equation}
  2^{[\dim {\mathfrak s}/2]}(p-q) = 2^{[\ell/2]}(p_0-q_0).
\end{equation}
Because of \eqref{eq:ncptroots},
\begin{equation}
\dim {\mathfrak s} = 2r + \ell,
\end{equation}
with $r$ the number of complex and noncompact imaginary positive
restricted roots, and $\ell$ the dimension of $A_c$. Therefore
\begin{equation}
  (p-q) = 2^{-r}(p_0-q_0),
\end{equation}
which is precisely the conclusion of the theorem.

So what if $\nu_c\ne 0$? In this case $D^2$ has at least some strictly
positive eigenvalues, meaning that $D$ has some real
eigenvalues. The proof of Lemma \ref{lemma:selfadjsig} would tell us
that we could compute $\Sig$ by restricting the form to these real
eigenspaces. The largest of these real eigenvalues we understand, but
the smaller ones are not easily accessible. So the proof strategy
appears to fail.

There are at least two ways out. The simplest is to work not with $G$
but with its commutator subgroup, a semisimple group. We already know
that an integral weight (like $(\lambda_c,\nu_c)$) must take real
values on the real span of the coroots. If $G$ is semisimple, this
real span of the coroots is
$$i{\mathfrak t}_c({\mathbb R}) + {\mathfrak a}_c({\mathbb R}).$$
Therefore the purely imaginary linear functional $\nu_c$ on $
{\mathfrak a}_c({\mathbb R})$ must be zero, and we are back in the
case \eqref{eq:kerD}.

A second (equivalent) method is to use the strongly integral weight
$$\chi = (0,-\nu_c).$$
This weight $\chi$ is the differential of a one-dimensional unitary
character ${\mathbb C}_\chi$ of $G({\mathbb R})$, so the signature of
$F(\gamma_c)$ is the same as the signature of $F(\gamma_c)\otimes
{\mathbb C}_\chi$. This latter representation has highest weight
$(\gamma_c, 0)$, so we are again in the case \eqref{eq:kerD}.

A third (still equivalent!) method would be to use not the
Dirac operator of \eqref{eq:D}, but one built from ${\mathfrak
  s}({\mathbb R}) \cap [{\mathfrak g},{\mathfrak g}]$. The reason we
did not do that is that there is a long history and literature
attached to Parthasarathy's Dirac operator; we preferred to use it and
to make this extra argument at the end.
\end{subequations} 
\end{proof}

\begin{proof}[Proof of Theorem \ref{thm:GLsig}]
  \begin{subequations}\label{se:GLsigpf}
    Now $G=GL(n,{\mathbb R})$, and
  \begin{equation}
    {\mathfrak s}({\mathbb R}) = \text{real symmetric matrices.}
  \end{equation}
We will treat the case $n=2m$ is {\em even}; the case of odd $n$ is
similar but slightly simpler, and we leave it to the reader. The
maximal compact torus is
\begin{equation}
  T_c({\mathbb R}) = SO(2)^m, \qquad X^*(T_c) \simeq {\mathbb Z}^m.
\end{equation}
The restricted root system is (see Table \ref{tab:resroots})
\begin{equation}
  R_{\text{\textnormal{res}}} = C_m, \quad
  R_{\text{\textnormal{cplx}}} = D_m, \quad
  R_{\text{\textnormal{imag}}} = A_1^m;
\end{equation}
all of the imaginary roots are noncompact. We use the standard
positive root system
\begin{equation}
  R^+_{\text{\textnormal{res}}} = \{e_j \pm e_k \mid 1\le j < k \le
  m\} \cup \{2e_j\}.
\end{equation}
Then we calculate
\begin{equation}
  \rho_G = (2m-1,2m-3,\ldots,1),\qquad \rho_K = (m-1,m-2,\ldots,0),
\end{equation}
\begin{equation}
  W^1 = W(A_1) = \{1,s_m\},
\end{equation}
with $s_m$ the reflection in the simple root $2e_m$.

The theorem concerns a restricted highest weight
\begin{equation}
  \lambda_c = (2\mu_1,\ldots,2\mu_m) = \left((\lambda_1 -
  \lambda_n),\cdots,(\lambda_m - \lambda_{m+1})\right).
\end{equation}
In the theorem we took all the $\mu_j$ to be integers, but Proposition
\ref{prop:fin} says that we can allow all the $\mu_j$ to be
half-integers as well.

By definition $\epsilon(1)=1$; Corollary \ref{cor:extrsig} says that
\begin{equation}
  \epsilon(s_m) = (-1)^{2\mu_m} = \begin{cases} 1 & \lambda_j \in
    {\mathbb Z} \\ -1 & \lambda_j \in {\mathbb Z} + 1/2.\end{cases}
\end{equation}

The two highest weights of $\widetilde K_0$ on the largest eigenspace
of $D^2$ are
\begin{equation}
  (2\mu_1,\ldots,2\mu_{m-1},\pm 2\mu_m) +
  (m,\ldots,2,\pm 1).
\end{equation}
These two representations of $\Spin(2m)$ differ by the outer
automorphism coming from $O(2m)$, so they have the same dimension. The
computation of the signature from Theorem \ref{thm:Gsig} is therefore
\begin{equation}
\Sig(\pi(\lambda)) = \begin{cases} 2\cdot\dim
  E(2\mu_1+m,\ldots,2\mu_m+1)/2^{m^2} & \lambda_j \in {\mathbb Z}\\ 0 &
  \lambda_j \in {\mathbb Z}+1/2.\end{cases}
\end{equation}
The dimension in this formula is calculated by the Weyl dimension
formula for $D_m$; the weight that must be inserted in the formula is
the highest weight plus $\rho_K$, which is
\begin{equation}
  (2\mu_1+2m-1,\ldots,2\mu_m+1).
\end{equation}
Now the Weyl dimension formula is a {\em homogeneous} polynomial of
degree $m^2-m$ (the number of positive roots for $D_m$; so
\begin{equation}
  \dim E(k\psi + (k-1)\rho_K) = k^{m^2-m} \dim E(\psi).
\end{equation}
If we apply this formula with $k=2$, we get
\begin{equation}
\Sig(\pi(\lambda)) = \begin{cases} 2\cdot\dim
  E(\mu_1+1/2,\ldots,\mu_m+1/2)/2^m & \lambda_j \in {\mathbb Z}\\ 0 &
  \lambda_j \in {\mathbb Z}+1/2.\end{cases}
\end{equation}
The first formula here is precisely Theorem \ref{thm:GLsig} (in case
$n$ is even).
\end{subequations} 
\end{proof}


\begin{bibdiv}
\begin{biblist}[\normalsize]

\bib{AvTV}{article}{
author={Adams, Jeffrey},
author={Leeuwen, Marc van},
author={Trapa, Peter},
author={Vogan, David A., Jr.},
title={Unitary representations of real reductive groups},
eprint={arXiv:1212.2192 [math.RT]},
}

\bib{Helgason}{book}{
author={S.~Helgason},
title={Differential Geometry, Lie Groups, and Symmetric Spaces},
publisher={Academic Press},
address={New York, San Francisco, London},
date={1978},
}

\bib{HKP}{article}{
   author={Huang, Jing-Song},
   author={Kang, Yi-Fang},
   author={Pand\v zi\'c, Pavle},
   title={Dirac cohomology of some Harish-Chandra modules},
   journal={Transform. Groups},
   volume={14},
   date={2009},
   number={1},
   pages={163--173},
}

\bib{Kmult}{article}{
   author={Kostant, Bertram},
   title={A cubic Dirac operator and the emergence of Euler number
   multiplets of representations for equal rank subgroups},
   journal={Duke Math. J.},
   volume={100},
   date={1999},
   number={3},
   pages={447--501},
}

\bib{Pdirac}{article}{
author={R.~Parthasarathy},
title={Dirac operator and the discrete series},
journal={Ann.\ of Math.},
volume={96},
date={1972},
pages={1--30},
}

\bib{Vgreen}{book}{
author={Vogan, David A., Jr.},
title={Representations of Real Reductive Lie Groups},
publisher={Birk\-h\"aus\-er},
address={Boston-Basel-Stuttgart},
date={1981},
}

\bib{IC4}{article}{
author={Vogan, David A., Jr.},
title={Irreducible characters of semisimple Lie
groups. IV. Char\-ac\-ter-mul\-ti\-plic\-i\-ty duality},
journal={Duke Math. J.},
volume={49},
number={4},
date={1982},
pages={943--1073},
}

\bib{atlas}{misc}{
title={Atlas of Lie Groups and Representations software},
note={\url{http://www.liegroups.org}},
year={2018}
}



\end{biblist}
\end{bibdiv}
\end{document}